\documentclass[11pt,reqno]{amsart}  

\usepackage{amscd} 
\usepackage{mathrsfs}
\usepackage{amsmath,amsbsy,amsthm,amsfonts,amssymb,esint}
\usepackage[unicode,psdextra]{hyperref}
\usepackage{graphicx}
\usepackage{tikz}
\usepackage{bm}
\usepackage{mathabx}
\usepackage{mathtools}
\usepackage{soul}
\usepackage{color}

\usepackage{float}

\usepackage{slashed}

\usepackage[sort&compress,capitalise,nameinlink]{cleveref}
\crefname{section}{\textsection}{\textsection}
\crefname{subsection}{\textsection}{\textsection}
\crefname{appendix}{\textsection}{\textsection}
\hypersetup{
    citecolor=red,
    linkcolor=blue}
\usepackage[capitalise]{cleveref} 
\usepackage{dsfont}
\usepackage{enumerate}

\DeclareMathAlphabet\mathbfcal{OMS}{cmsy}{b}{n}

\newtheorem{theorem}{Theorem}[section]
\newtheorem{corollary}{Corollary}[section]
\newtheorem{lemma}{Lemma}[section]
\newtheorem{proposition}{Proposition}[section]
\theoremstyle{definition}
\newtheorem{remark}{Remark}[section]

\newcommand{\Con}{\mathcal{C}}
\newcommand{\psili}{\psi_{\Nh}}
\newcommand{\psih}{\hat{\psi}}
\newcommand{\Nh}{\Xi}
\newcommand{\zet}{z_*}
\newcommand{\DD}{H}
\newcommand{\DDz}{\DD_0}

\newcommand{\lz}{\lambda_{0}}
\newcommand{\DDet}{\DD_{\eta,\tau}}
\newcommand{\DDket}{\DD[k]_{\eta,\tau}}
\newcommand{\Fou}{\mathscr{F}}
\newcommand{\Hil}{H}
\newcommand{\xv}{{\bm{x}}}

\newcommand{\kv}{{\bm{k}}}
\newcommand{\vv}{{\bm{v}}}
\newcommand{\nv}{{\bm{n}}}
\newcommand{\tv}{{\bm{t}}}
\newcommand{\siv}{{\bm\sigma}}

\newcommand{\one}{\mathds{1}}

\newcommand{\dom}{\mathscr{D}}

\newcommand{\sgn}{\mathrm{sgn}}

\newcommand{\beq}{\begin{equation}}
\newcommand{\eeq}{\end{equation}}

\newcommand{\UU}{\mathcal{U}}

\newcommand{\Z}{\mathbb{Z}}
\newcommand{\CC}{\mathbb{C}}
\newcommand{\R}{\mathbb{R}}
\newcommand{\N}{\mathbb{N}}

\newcommand{\TD}{\mathcal{T}^D}

\newcommand{\bdm}{\begin{displaymath}}
\newcommand{\edm}{\end{displaymath}}
\newcommand{\bdn}{\begin{eqnarray}}
\newcommand{\edn}{\end{eqnarray}}
\newcommand{\bay}{\begin{array}{c}}
\newcommand{\eay}{\end{array}}
\newcommand{\ben}{\begin{enumerate}}
\newcommand{\een}{\end{enumerate}}
\newcommand{\beqn}{\begin{eqnarray}}
\newcommand{\eeqn}{\end{eqnarray}}

\newcommand{\be}{\begin{equation}}%
\newcommand{\ee}{\end{equation}}

\renewcommand{\geq}{\geqslant}

\title[On two-dimensional Dirac operators with critical delta-shell interactions]{On two-dimensional Dirac operators\\with critical delta-shell interactions}

\author{William Borrelli}
\address{Dipartimento di Matematica, Politecnico di Milano, P.zza Leonardo da Vinci, 32, 20133, Milano, Italy}
\email{william.borrelli@polimi.it}
\urladdr{}

\author{Pietro Carimati}
\address{Dipartimento di Matematica, Politecnico di Milano, P.zza Leonardo da Vinci, 32, 20133, Milano, Italy}
\email{pietro.carimati@mail.polimi.it}

\author{Davide Fermi}
\address{Dipartimento di Matematica, Politecnico di Milano, P.zza Leonardo da Vinci, 32, 20133, Milano, Italy
and Istituto Nazionale di Fisica Nucleare, Sezione di Milano, Italy}
\email{davide.fermi@polimi.it}
\urladdr{https://fermidavide.com}

\begin{document}

\begin{abstract} 
We study two-dimensional Dirac operators with singular interactions of electrostatic and Lorentz-scalar type, supported either on a straight line or a circle. For certain critical values of the interaction strengths, the essential spectrum of such operators comprises an isolated point lying within the mass gap. We clarify the nature of this point in both geometries. For the straight line model, this point is known to be an eigenvalue of infinite multiplicity, and we provide a detailed analysis of the corresponding eigenfunctions. By contrast, in the case of a circle, we show that the said point is not itself an eigenvalue, but rather an accumulation point of a double sequence of simple eigenvalues. In view of the high degree of symmetry of the configurations under analysis, this behavior is unexpected and our findings lead us to formulate some conjectures concerning critical singular interactions supported on generic smooth curves.
\end{abstract}


\maketitle
\section{Introduction and main results}\label{sec:intro}
In this paper we study two-dimensional Dirac operators formally given by
\begin{equation}\label{eq:fD}
-i(\sigma_1\partial_x+\sigma_2\partial_y)+m\,\sigma_3+(\eta\, \sigma_0+\tau\, \sigma_3)\delta_\Sigma\,,
\end{equation} 
where $\sigma_0$ is the $2\times2$ identity matrix, $(\sigma_j)_{j=1,2,3}$ are the Pauli matrices
\begin{equation}\label{eq:paulim}
\sigma_1=\begin{pmatrix} 0 & 1 \\ 1 & 0 \end{pmatrix},\qquad 
\sigma_2=\begin{pmatrix} 0 & -i \\ i & 0 \end{pmatrix},\qquad
\sigma_3=\begin{pmatrix} 1 & 0 \\ 0 & -1 \end{pmatrix},
\end{equation}
while $m>0$ is the mass parameter and $\delta_\Sigma$ is the Dirac $\delta$-distribution supported on a smooth simple curve $\Sigma\subset\R^2$. The terms $\eta\, \delta_\Sigma$ and $\tau\,\sigma_3 \,\delta_\Sigma$ in \eqref{eq:fD} are usually referred to as \emph{electrostatic} and \emph{Lorentz-scalar} singular interactions, respectively, with coupling parameters $\eta,\tau\in\R$ quantifying their strength.

More generally, one may consider additional terms in \eqref{eq:fD} of the form $(\tv \cdot \siv)\,\lambda\,\delta_\Sigma$ and  $(\nv \cdot \siv)\,\omega\,\delta_\Sigma$, with $\lambda,\omega\in\R$, where $\bm{t}$ and $\bm{n}$ are the tangential and normal vectors along the curve $\Sigma$, respectively. These additional terms represent interactions of magnetic type, first introduced in \cite{CLMT22}. However, in this work we are interested only in the case $\lambda = \omega = 0$. Here and in what follows, the notation
 \[
 \vv \cdot \siv = v_1 \sigma_1 + v_2 \sigma_2 \,,
 \]
 for an arbitrary vector $\bm v=(v_1,v_2) \in \R^2$ is employed.

Operators of the form \eqref{eq:fD} have attracted a considerable interest in recent years. Generally speaking, Dirac operators with local boundary conditions in 2d have been intensively studied, in connection with the effective description of graphene and related materials, see \cite{AB08, Be08, BFSV17, BBKO22, CLMT22, BCF24, GUV25} and references therein, while the three-dimensional analogue has been motivated by the so-called MIT bag model \cite{CJJTW74} for hadrons confinement in QCD, which then led to consider more general boundary conditions in the literature, for which we refer, \emph{e.g.}, to \cite{DES89,AMV14,OP21}.

The formal operator \eqref{eq:fD} can be realized as a proper self-adjoint Hamiltonian on $L^2(\R^2\,;\CC^2)$, see \emph{e.g.} \cite{BHOP20,BMSS24}. Following these references, for a given smooth simple curve $\Sigma\subseteq\R^2$, we denote by $\Omega_\pm$ the two disjoint open regions in which $\Sigma$ divides the plane $\R^2$, namely, 
\[
\R^2 = \Omega_+ \cup \Omega_- \cup \Sigma\,,\qquad
\Omega_+ \cap \Omega_- = \varnothing\,, \qquad
\partial\Omega_+ = \partial\Omega_- = \Sigma\,.
\]
If the curve $\Sigma$ is closed, we assume $\Omega_+$ to be the bounded region by convention. Accordingly, we shall refer to the Hilbert spaces
 \[
 \Hil(\siv,\Omega_\pm):= \big\{f \in L^2(\Omega_\pm;\CC^2) \;\big|\; (\siv \cdot \nabla)f\in L^2(\Omega_\pm;\CC^2)\big\}\,.
 \]
 It can be proved that functions $f_\pm\in \Hil(\siv,\Omega_\pm)$ admit Dirichlet traces $\TD_\pm f_\pm\in H^{-1/2}(\Sigma\,;\CC^2)$. Accordingly, the self-adjoint Hamiltonian operator $\DDet$ matching the formal expression \eqref{eq:fD} can be defined as follows, in terms of suitable boundary conditions on $\Sigma$:
\begin{eqnarray}
& \hspace{-4cm}\dom(\DDet) := \big\{\psi = \psi_+\oplus\psi_-\in \Hil(\siv,\Omega_+)\oplus \Hil(\siv,\Omega_-)\quad \big| \nonumber \\
& \hspace{4cm} -i(\nv \cdot \siv)(\TD_+\psi_+-\TD_-\psi_-)=\tfrac{1}{2} (\eta\, \sigma_0+\tau\,\sigma_3)(\TD_+\psi_++\TD_-\psi_-)\big\}\,, \label{eq:Ddom} \\ 
& \DDet \psi:=(-i(\sigma_1\partial_x+\sigma_2\partial_y)+m\,\sigma_3)\psi_+\oplus(-i(\sigma_1\partial_x+\sigma_2\partial_y)+m\,\sigma_3)\psi_- \,. \label{eq:Daction}
\end{eqnarray}
Here and in the following, we choose $\nv$ to be the unit normal to $\Sigma$ pointing outwards of $\Omega_+$. 

The case of a \emph{closed} curve $\Sigma$ appears to be the most studied and various results on self-adjointness and spectral properties of the Hamiltonian $\DDet$ are available in the literature. Very recently, the extension to curves with low regularity, namely Lipschitz continuous curves such as closed polygons, was discussed in \cite{BPZ25}. On the other hand, the case of an unbounded simple curve $\Sigma$ has been considered only in special cases \cite{CJJTW74,R21,BHT23,BEHT25}. In both cases, it is well understood that $\DDet$ displays unusual spectral features, compared to the case of regular potentials (\emph{i.e.} given by multiplication operators), in the so-called \emph{critical case}
 \begin{equation}\label{eq:critical}
\eta^2-\tau^2=4\,.
\end{equation}
When $\eta^2-\tau^2\neq 4$, that is, in the non-critical regime of parameters, the essential spectrum is given by
\[
\sigma_{\mathrm{ess}}(\DDet)=(-\infty, -m]\cup[m,+\infty)\,,
\]
while the discrete spectrum inside the gap $(-m,m)$ has finite cardinality. This is what typically happens for regular potential perturbations, under quite general assumption. On the contrary, it is known that in the critical regime \eqref{eq:critical} the essential spectrum also contains an isolated point, namely,
\begin{equation}\label{eq:critess}
\sigma_{\mathrm{ess}}(\DDet)=(-\infty, -m]\cup \left\{-\tfrac{\tau}{\eta}\,m\right\}\cup[m,+\infty)\,,
\end{equation}
together with discrete eigenvalues in the gap.

Singular interactions on an unbounded simple curve $\Sigma$, to our knowledge, have been considered in a systematic way in the case of a straight line \cite{BHT23}, for curves with straight ends \cite{BEHT25}, while in \cite{R21} the general case has been studied. The quite detailed description provided in \cite{BHT23} when $\Sigma$ is a straight line shows that the spectrum may differ from the case of a closed curve. In particular, in the non-critical case the essential spectrum can be fully determined, and it is of the form 
\[
\sigma_{\mathrm{ess}}(\DDet)=(-\infty, -m]\cup[m,+\infty)\cup \mathcal I_+\cup\mathcal I_-\,,
\]
where $\mathcal I_\pm\subseteq\R$ are suitable intervals that can be explicitely characterized. As for the case of a closed curve, assuming \eqref{eq:critical} one can prove that the essential spectrum is given by \eqref{eq:critess}. Moreover, in that case there are no additional points in the spectrum, that is $\sigma(\DDet)=\sigma_{\mathrm{ess}}(\DDet)$, and the isolated point 
\begin{equation}\label{eq:z}
\zet=-\frac{\tau}{\eta}\,m \in(-m,m)
\end{equation}
is known to be an eigenvalue of infinite multiplicity. When the unbounded curve has straight ends, the localized curvature can create additional bound states. This has been proved in \cite{BEHT25} in the non-relativistic limit assuming $\eta=\tau$, and in the purely Lorentz scalar case, $\eta=0$.

Three-dimensional analogues of the models described above were studied in \cite{OV18,BH20,Be22,BMSS24}. In this context, it was recently proved in \cite{BP24} that critical combinations of electrostatic and Lorentz scalar interactions supported on a compact smooth surface give rise to an additional band of essential spectrum inside the spectral gap. Its location and size can be characterized explicitly in terms of the coupling parameters and the principal curvatures of the supporting surface.

It is further worth noting that it is known in the literature that examples of anomalous spectral features similar to those mentioned above also occurr for other type of operators. Indeed, this is for instance the case for indefinite Laplacians, namely operators obtained as self-adjoint realizations of symmetric expressions of the form $-\operatorname{div} \left( h \nabla\,\cdot\right)$, where $h$ is a piecewise constant function that changes sign across the space domain. Operators of this type appear in the study of metamaterials with negative refractive index, producing anomalous localized resonances and cloaking effects.
In this setting, it is known that for critical choices of $h$ the point $0$ belongs to the essential spectrum, and that its nature depends on the geometry of the domains where $h$ takes opposite signs: it can be either an eigenvalue of infinite degeneracy or an accumulation point of eigenvalues \cite{BK18,CPP19,Pa19,Fa25}.
Very similar effects have been recently shown to arise in the spectral theory of Maxwell operators as well \cite{Fe25}.

\medskip

In this paper we prove that the nature of the spectral point \eqref{eq:z} for 2D Dirac operators with critical $\delta$-shell interactions can change, depending on the geometry of the curve. To this aim, we consider two case studies, namely, a straight line and a circle of radius $R>0$. The paper is organized as follows.
  
In \cref{thm:mainline} and \cref{cor:eigline} we provide further qualitative and quantitative insight into the eigenfunctions when $\Sigma$ is a straight line, thereby extending the analysis in \cite{BHT23}.  
We then show in \cref{thm:maincirc} that, when $\Sigma$ is a circle, the point $\zet$ defined in \eqref{eq:z} is not itself an eigenvalue, but rather an accumulation point for a double sequence of eigenvalues in $(-m,m)\,\setminus\,\{-\tau m/\eta\}$, converging to $\zet$ both from the left and from the right. We also compute explicitly the associated normalized eigenfunctions. In \cref{cor:asef}, we examine the asymptotic behavior of the eigenvalues and eigenfunctions as the total angular momentum increases. Our results suggest that the observed features may be universal for singular interactions supported on closed curves (see \cref{rmk:universal}). Along the way, we outline three conjectures that we leave open for future investigations.

\subsection{The free Dirac operator}\label{sec:freeD}
Let us recall that the two-dimensional free Dirac operator is the self-adjoint operator on $L^2(\R^2,\CC^2)$ defined as
\begin{equation}\label{eq:D}
\DDz := -i\,\siv \cdot \nabla + m\sigma_3\,, \qquad \dom(\DDz) := H^1(\R^2,\CC^2)\,.
\end{equation}
This is unitarily equivalent, via Fourier transform, to the multiplication by the matrix-valued function
\begin{equation}\label{eq:hatD}
\DDz[\kv] = \begin{pmatrix} m & k_1-i k_2 \\ k_1+ik_2 & -m \end{pmatrix} ,
\end{equation}
where $\kv = (k_1,k_2)\in\R^2$ is the momentum coordinate. More precisely, using the unitary map
\[
	\Fou : L^2(\R^2;\CC^2) \to L^2(\R^2;\CC^2)\,, \qquad
	(\Fou \psi)(\kv) := \frac{1}{2\pi} \int_{\R^2} e^{-i \kv \cdot \xv}\, \psi(\xv)\,d\xv\,,
\]
the following direct integral decomposition holds
\[
\Fou \DDz \Fou^{-1} = \int^{\oplus}_{\R^2}\DDz[\kv]\,d\kv\,, \qquad
\dom(\Fou \DDz \Fou^{-1}) = \Fou H^1(\R^2,\CC^2)\,.
\]
The spectrum of $\DDz$ is purely absolutely continuous. In fact, it is spanned by the eigenvalues of the matrix \eqref{eq:hatD}, namely,
\[
\sigma(\DDz)=\sigma_{\mathrm{ac}}(\DDz) 
= \big\{\pm\lz(\kv)=\pm\sqrt{m^2+\vert \kv\vert^2}\;\big|\; \kv\in\R^2\big\}
=(-\infty,-m]\cup[m,+\infty)\,.
\]

\subsection{Singular interactions on a straight line}\label{sec:Dline}
Consider the case where $\Sigma$ is a straight line. Without loss of generality, we may fix
\begin{equation}\label{eq:line}
\Sigma = \Sigma_L := \big\{(0,y)\in\R^2\;|\; y\in\R\big\}\,.
\end{equation}
On account of the invariance under translations in the second variable, we may use the unitary \emph{partial} Fourier transform
\[
	\Fou_2 : L^2(\R^2,dxdy;\CC^2) \to L^2(\R^2,dxdk;\CC^2)\,, \qquad
	(\Fou_2 \psi)(x,k) := \frac{1}{\sqrt{2\pi}} \int_{\R} e^{-i k y}\, \psi(x,y)\,dy\,,
\]
to obtain the following decomposition for the free Dirac operator:
\[
\Fou_2 \DDz \Fou^{-1}_2=\int^\oplus_{\R}\DDz[k]\,dk\,,
\]
where the fiber operator at fixed $k \in\R$ is given by
\begin{equation}\label{eq:fibfree}
\DDz[k]:=\sigma_1(-i\partial_x) + k \sigma_2 + m\sigma_3\,, \qquad 
\dom\big(\DDz[k]\big):= H^1(\R,dx\,;\CC^2)\,.
\end{equation}
The spectrum of the latter is purely absolutely continuous and reads
\[
\sigma\big(\tilde{h}_0[k]\big)=\sigma_{\mathrm{ac}}\big(\tilde{h}_0[k]\big)=\left(-\infty, -\sqrt{m^2+k^2}\,\right]\cup\left[\sqrt{m^2+k^2}+\infty\right).
\]

The definition given in \eqref{eq:Ddom} and \eqref{eq:Daction} clearly applies to the present context, with $\Omega_\pm := \{(x,y)\in \R^2\,|\,\pm x > 0\}$ and $\nv = (-1,0)$. This corresponds to the fiberwise addition of singular interactions, namely,
\begin{eqnarray}
& \hspace{-4.cm} \dom(\DDket):= \Big\{\psih=\psih_+\oplus\psih_-\in H^1(\R_+\,;\CC^2)\oplus H^1(\R_-\,;\CC^2)\quad \mbox{s.t.} \nonumber\\
& \hspace{4.cm} i\sigma_1\big(\psih_+(0^+)-\psih_-(0^-)\big)=\tfrac{1}{2}(\sigma_0\eta+\sigma_3\tau)\big(\psih_+(0^+)+\psih_-(0^-)\big) \Big\}\,, \label{eq:fibline} \\ 
& \DDket\psih :=(-i\sigma_1\partial_x+\sigma_2 k + m\sigma_3)\psih_-\oplus (-i\sigma_1\partial_x+\sigma_2 k + m\sigma_3)\psih_+\,.
\end{eqnarray}

Let us point out that, whenever $\eta^2 - \tau^2 \neq -4$ (notice the minus sign on the right-hand side), the boundary condition in \eqref{eq:fibline} can be equivalently rephrased as
\begin{equation}\label{eq:bcLam}
\psih_+(0^+) = \Lambda_{\eta,\tau}\, \psih_-(0^-)\,,
\end{equation}
where
\begin{equation}\label{eq:Lamdef}
\Lambda_{\eta,\tau} := \left[i\sigma_1 - \tfrac{1}{2}(\eta \sigma_0+\tau \sigma_3)\right]^{-1} \left[i\sigma_1 + \tfrac{1}{2}(\eta \sigma_0+\tau \sigma_3)\right] .
\end{equation}

Working backwards, the full Hamiltonian $\DDet$, matching the definition in \eqref{eq:Ddom} and \eqref{eq:Daction}, can be expressed as
\begin{equation}\label{eq:Hline}
\Fou_2 \DDet \Fou_2^{-1} = \int^\oplus_{\R} \DDket\,dk\,.
\end{equation}
We refer the reader to \cite{BHT23} for the properties of \eqref{eq:Hline} with arbitrary values of the parameters $\eta,\tau \in\R$. Here, we limit ourselves to recall the following result for the critical case, see \cite[Thm. 6.2]{BHT23}.

	\begin{proposition}\label{prop:line}
		Let $\Sigma = \Sigma_L$ be as in \eqref{eq:line} and let $\eta,\tau \in \R$ with $\eta^2-\tau^2=4$. Then, the spectrum of $\DDet$ is given by
			\[
				\sigma(\DDet) = \sigma_{\mathrm{ac}}(\DDet) \cup\{\zet\}\,,
			\]
where $\sigma_{\mathrm{ac}}(\DDet)=(-\infty,m]\cup[m,+\infty)$ and
			\[
				\zet=-\frac{\tau}{\eta}\,m \in (-m,m)
			\]
		is an isolated eigenvalue of infinite multiplicity.
	\end{proposition}

The eigenfunctions for a generic delta-shell interaction concentrated on a straight line were firstly considered in \cite{BHT23}. In the forthcoming \cref{thm:mainline}, we provide a more detailed characterization of the eigenspace associated to the isolated eigenvalue $\zet$ and, in the subsequent \cref{cor:eigline}, we examine some notable features of the corresponding eigenfunctions.

	\begin{theorem}\label{thm:mainline}
		Let $\Sigma = \Sigma_L$ be as in \eqref{eq:line} and let $\eta,\tau \in \R$ satisfy $\eta^2-\tau^2=4$. Then, the eigenspace associated to the infinitely degenerate eigenvalue $\zet$ consists of functions of the form
		\begin{equation}\label{eq:eignfunexp}
			\psili(x,y) = \tfrac{1}{2\sqrt{\pi}} \big[\theta(x)\,\Lambda_{\eta,\tau} + \theta(-x)\,\sigma_0 \big] \int_{\R}\!\! dk
				\begin{pmatrix}
					\sqrt{\left(1-\frac{\tau}{\eta}\right)\left(\sqrt{k^2 + \frac{4m^2}{\eta^2}} + k\right)} \\
					-i\,\sqrt{\left(1+\frac{\tau}{\eta}\right)\left(\sqrt{k^2 + \frac{4m^2}{\eta^2}} - k\right)} 
				\end{pmatrix}
			\Nh(k)\, e^{i k y -\sqrt{k^2 +\frac{4m^2}{\eta^2}}\,\vert x\vert}\,.
		\end{equation}
		where $\Lambda_{\eta,\tau}$ is the $2\times 2$ matrix defined in \eqref{eq:Lamdef} and $\Nh \in L^2(\R,dk)$ is a normalization factor fulfilling
		\begin{equation}\label{eq:normeq}
			\|\psili\|_{L^2(\R^2,\,dxdy;\,\CC^2)} = \|\Nh\|_{L^2(\R,\,dk)}\,.
		\end{equation}
	\end{theorem}

	\begin{remark}
		In the critical case under analysis, the definition \eqref{eq:Lamdef} reduces to
		\begin{equation}\label{eq:Lamcrit}
			\Lambda_{\eta,\tau} = \begin{pmatrix}
				0 & -\frac{i}{2} (\eta - \tau) \\
				-\frac{i}{2} (\eta + \tau) & 0
			\end{pmatrix}.
		\end{equation}
		It is noteworthy that $\Lambda_{\eta,\tau}^2 = - \one$, which means that $\Lambda_{\eta,\tau}$ yields a complex structure on $\CC^2$. 		This is an exceptional feature which never occurs for non-critical values of $\eta,\tau$.
	\end{remark}

	\begin{remark}
		The critical condition $\eta^2 - \tau^2 = 4$ entails both $|\eta| > 2$ and $|\tau/\eta| < 1$, which ensures that the arguments of all the square roots appearing in \eqref{eq:eignfunexp} are indeed positive definite. On the other hand, the identity \eqref{eq:normeq} makes evident that any $\psili$ of the form \eqref{eq:eignfunexp} is an eigenstate with unit norm in $L^2(\R^2;\CC^2)$ if and only if $\|\Nh\|_{L^2(\R,\,dk)} = 1$. 
	\end{remark}

	\begin{remark}
		As a direct consequence of \cref{thm:mainline}, the eigenspace associated to the eigenvalue $\zet$ is generated by functions $\psili$ of the form \eqref{eq:eignfunexp}, with the normalization factor $\Nh$ ranging in any Hilbert basis of $L^2(\R,dk)$. For example, one may choose $\Nh$ within the set of (normalized) energy eigenfunctions of the one dimensional harmonic oscillator, namely,
		\begin{equation*}
			\Nh \in \{b_n\}_{n \in \N}\,, \qquad b_n(k) := \tfrac{(2/\pi)^{1/4}}{\sqrt{2^n n!}}\,H_n\big(\sqrt{2}\, k\big)\,e^{-k^2} ,
		\end{equation*}
		where $H_n$ ($n \in \N$) are the Hermite  polynomials \cite[\S 18.3]{NIST}.
		To say more, the normalization factor $\Nh$ is allowed to incorporate a pure phase term of the form $e^{i k y_0}$. This corresponds to a translation by $y_0$ of the coordinate $y$, namely $y \mapsto y + y_0$, which matches in turn a shift of the bound state along the line $\Sigma$. For illustrative purposes, \cref{fig:modplots1} and \cref{fig:modplots2} display the plots of the squared modulus $|\psili(x,y)|^2$, for two different choices of the interaction parameters $\eta,\tau$ and of the form factor $\Nh$.
	\end{remark}
	\begin{figure}[t!]
    \begin{minipage}[b]{0.4\textwidth}
        \centering
        \includegraphics[width=\textwidth]{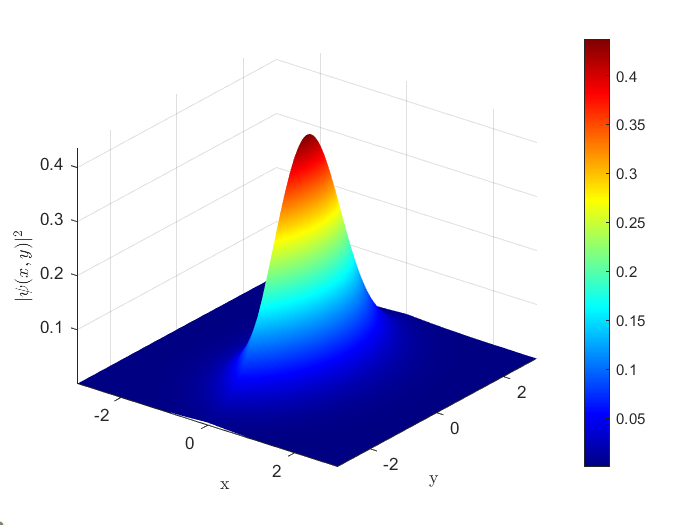}
    \end{minipage}
    \begin{minipage}[b]{0.4\textwidth}
        \centering
        \includegraphics[width=\textwidth]{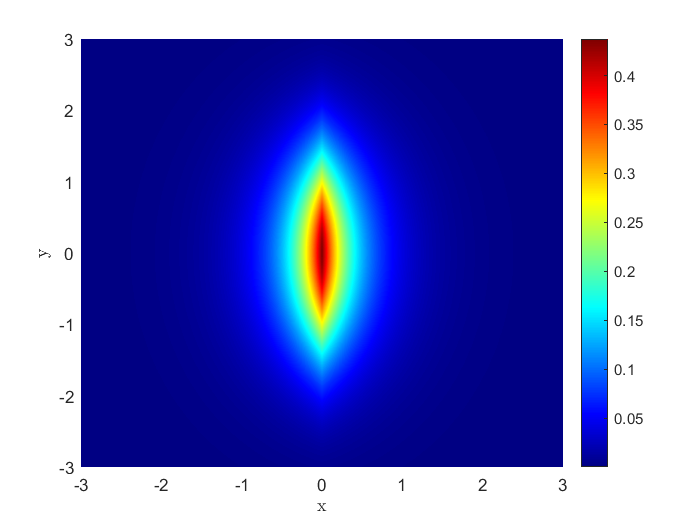}
    \end{minipage}
    \caption{
			Plot of $\vert \psili(x,y)\vert^2$ for $\eta=2$, $\tau=0$ and $\Xi(k) = b_0(k) = \big(\tfrac{2}{\pi}\big)^{1/4} e^{-k^2}$.
		}
    \label{fig:modplots1}
	\end{figure}
	\begin{figure}[t!]
      \begin{minipage}[b]{0.4\textwidth}
        \centering
        \includegraphics[width=\textwidth]{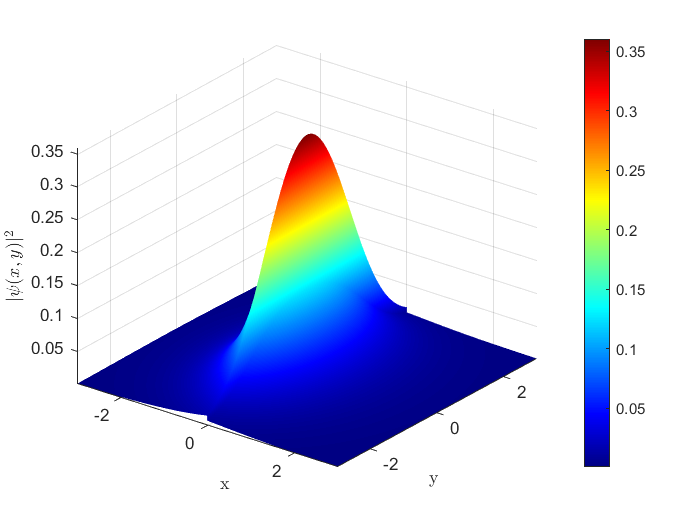}
    \end{minipage}
    \begin{minipage}[b]{0.4\textwidth}
        \centering
        \includegraphics[width=\textwidth]{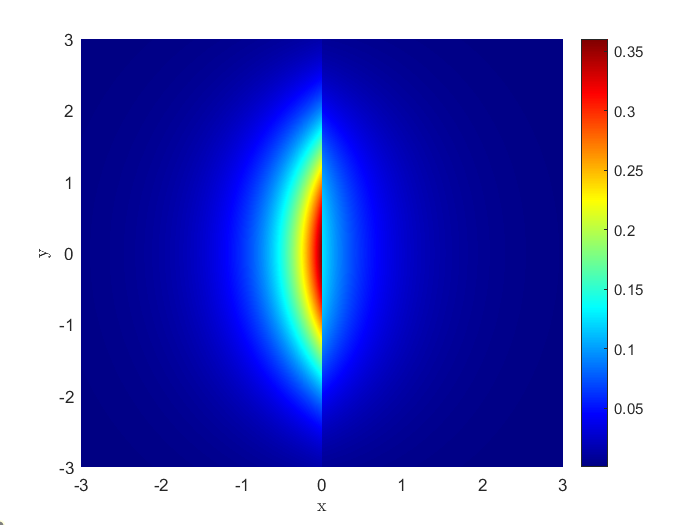}
    \end{minipage}
    \caption{
		Plot of $\vert \psili(x,y)\vert^2$ for $\eta=\sqrt{13}$, $\tau=-3$ and $\Xi(k)=\frac{2}{\sqrt 5}\!\left(\frac{2}{\pi}\right)^{1/4}(k+1)\,e^{-k^2}$.
		}
    \label{fig:modplots2}
	\end{figure}

\begin{remark}
The asymptotic behavior of the eigenstates for large values of the spatial coordinates can be obtained starting from the explicit expression \eqref{eq:eignfunexp} by standard stationary phase techniques.
In particular, if $\Nh$ belongs to the Schwartz class, then $\psili$ decays faster than any polynomial as $y \to \pm \infty$, for any fixed $x \in \R \setminus \{0\}$.
On the other side, the behavior of $\psili(x,y)$ for fixed $y \in \R$ and $x \to \pm \infty$, can be analyzed using the Laplace method, see \emph{e.g.} \cite[Thm. 8.1]{Olv}.
For example, assuming $\Nh$ to be continuous in a neighborhood of $k = 0$, we get
\begin{equation*}
		\psili(x,y) \sim  \big[ \theta(x)\, \Lambda_{\eta,\tau} + \theta(-x) \sigma_0  \big]  \begin{pmatrix}
		\sqrt{1-\frac{\tau}{\eta}} \vspace{0.1cm}\\ -i\,\sqrt{1+\frac{\tau}{\eta}} \end{pmatrix} \frac{\sqrt{2}\,m}{|\eta|}\,\Nh(0)\, \frac{e^{- \frac{2m}{|\eta|} |x|}}{|x|^{1/2}}\, , \qquad \mbox{for } \vert x \vert \to \infty\,.
\end{equation*}
where we employed the standard notation $f(x) \sim g(x)$ to indicate that $\lim_{|x| \to \infty} f(x)/g(x) = 1$. Higher order terms of the expansion could be obtained with stronger regularity hypotheses on $\Nh$.
\end{remark}

\begin{remark}
The dependence on the mass parameter is actually trivial. In fact, we can fix the length scale to be $m^{-1}$ and work in dimensionless units setting 
\[ 
\hat{\xv} := m\, \xv\,,\; \hat{\kv} := m^{-1}\kv \qquad \mbox{and} \qquad \hat{\Nh}(\hat{\kv}) = m^{3/2}\, \Nh(m \hat{\kv}) \,. 
\]
Accordingly, the eigenfunction $\psili(x,y)$ in \eqref{eq:eignfunexp} becomes
\begin{equation}\label{eq:eignfunexpdim}
		\hat{\psi}_{\Nh}(\hat{x},\hat{y}) = \frac{1}{2\sqrt{\pi}}\, \big[ \theta(\hat{x})\, \Lambda_{\eta,\tau} + \theta(-\hat{x}) \sigma_0 \big] \int_{\R}\! d\hat{k}  \begin{pmatrix} \sqrt{\left(1-\frac{\tau}{\eta}\right)\!\left(\sqrt{\hat{k}^2 + \frac{4}{\eta^2}} + \hat{k}\right)} \vspace{0.1cm}\\ -i\,\sqrt{\left(1+\frac{\tau}{\eta}\right)\!\left(\sqrt{\hat{k}^2 + \frac{4}{\eta^2}} - \hat{k}\right)}  \end{pmatrix} \hat{\Nh}(\hat{k})\,e^{i \hat{k} \hat{y} - \sqrt{\hat{k}^2 + \frac{4}{\eta^2}}\,\vert \hat{x} \vert}.
\end{equation}
\end{remark}
\medskip

We now proceed to compute the expectation values and standard deviations of three key observables for the eigenstates characterized in \cref{thm:mainline}, namely the spin $\sigma_3$, the position $\xv = (x,y)$, and the velocity $\vv = (v_x,v_y)$. While the first two observables are standard, let us recall that in the Dirac setting the velocity operator is defined as, see \cite{Th92},
\[
	\vv := i [\DDz,\xv]\,,
\]
and an elementary calculation of the commutator yields
\[
	v_x = \sigma_1\,, \qquad v_y = \sigma_2\,.
\]
Observe that $\sigma_3,v_x,v_y$ are bounded self-adjoint operators on $L^2(\R^2,\CC^2)$, whereas the position components are unbounded with natural self-adjointness domains
\[
	\dom(x) = \{\psi \in L^2(\R^2,\CC^2)\,|\, x\,\psi \in L^2(\R^2,\CC^2)\}\,, \qquad 
	\dom(y) = \{\psi \in L^2(\R^2,\CC^2)\,|\, y\,\psi \in L^2(\R^2,\CC^2)\}\,.
\]

For any observable $\mathcal{O}$ among those listed above and for any normalized state $\psi_{\eta,\tau}$, we set
\[
	\langle \mathcal{O} \rangle_{\Nh} := \langle \psili, \mathcal{O} \,\psili \rangle\,, \qquad
	\langle (\Delta\mathcal{O})^2 \rangle_{\Nh} := \|\mathcal{O}\psili\|^2 - \langle \mathcal{O} \rangle_{\Nh}^2\,.
\]

	\begin{corollary}\label{cor:eigline}
		Assume the hypotheses of \cref{thm:mainline} and let $\psili$ be of the form \eqref{eq:eignfunexp}, with $\Nh \in L^2(\R,dk)$. Then, $\psili \in \dom(x)$ and
		\begin{equation}
			\psili \in \dom(y) \qquad \mbox{if and only if} \qquad \Nh, \partial_k \Nh \in L^2(\R,dk)\,.
		\end{equation}
		Moreover, for any admissible $\Nh$ with $\|\Nh\| = 1$, the following identities hold true:
		\begin{equation}\label{eq:s3}
			\langle \sigma_3 \rangle_{\Nh} = -\tfrac{\tau}{\eta}\,, \qquad
			\langle (\Delta \sigma_3)^2 \rangle_{\Nh} = \tfrac{4}{\eta^2}\,;
		\end{equation}
		\begin{equation}\label{eq:xavg}
			\langle x \rangle_{\Nh} = \tfrac{\tau}{2\eta}\,\Big\langle \Nh,\tfrac{k}{k^2 + 4 m^2/\eta^2}\,\Nh\Big\rangle\,, \quad\;
			\langle (\Delta x)^2 \rangle_{\Nh} = \tfrac{1}{2}\Big\| \tfrac{1}{\sqrt{k^2 + 4 m^2/\eta^2}}\,\Nh \Big\|^2 - \tfrac{\tau^2}{4\eta^2}\! \left(\Big\langle \Nh,\tfrac{k}{k^2 + 4 m^2/\eta^2}\,\Nh\Big\rangle\right)^2 ;
		\end{equation}
		\begin{equation}
			\langle y \rangle_{\Nh} = \langle \Nh, i \partial_k \Nh \rangle\,, \qquad
			\langle (\Delta y)^2 \rangle_{\Nh} = \|i \partial_k \Nh\|^2 + \frac{1}{4} \left\|\tfrac{1}{\sqrt{k^2 + 4 m^2/\eta^2}}\,\Nh\right\|^2 - \big(\langle \Nh, i \partial_k \Nh \rangle\big)^2\,;
		\end{equation}
		\begin{equation}\label{eq:vxvy}
			\langle v_x \rangle_{\Nh} = \langle v_y \rangle_{\Nh} = 0\,, \qquad
			\langle (\Delta v_x)^2 \rangle_{\Nh} = \langle (\Delta v_y)^2 \rangle_{\Nh} = 1\,.
		\end{equation}
	\end{corollary}

	\begin{remark}
		The first relation in \eqref{eq:s3} shows that eigenstate $\psili$ has zero average spin if and only if $\tau = 0$, while the second identity in the same equation makes evident that the spin uncertainty can be made arbitrarily small by taking $\eta$ sufficiently large.
On the other hand, let us notice that the first relation in \eqref{eq:xavg} can be written explicitly as
		\[
			\langle x \rangle_{\Nh} = \frac{\tau}{2\eta} \int_{\R} dk\; \frac{k}{k^2 + 4 m^2/\eta^2}\,|\Nh(k)|^2\,.
		\]
		This makes evident that the eigenfunction $\psili$ is centered on the line $\Sigma$ only if either $\tau = 0$ (whence, $\eta = \pm 2$) or $|\Nh(k)|$ is an even function of $k \in \R$. In this connection, we note that the plots in \cref{fig:modplots1} correspond to a case where $\langle x \rangle_{\Nh} = 0$, whereas those in \cref{fig:modplots2} illustrate a case in which $\langle x \rangle_{\Nh} \neq 0$.
		Finally, Eq. \eqref{eq:vxvy} implies that the eigenfunction $\psili$ has zero average velocity, as expected for a bound state.
	\end{remark}

\subsection{Singular interactions on a circle}\label{sec:Dcircle}
Let us now consider the case where $\Sigma$ is a circle of radius $R>0$, that we take centered at the origin:
\begin{equation}\label{eq:circle}
\Sigma = \Sigma_R:= \big\{(x,y)\in\R^2\,\big|\, x^2+y^2=R^2\big\}\,,
\end{equation}
In agreement with the established convention, we assume $\Omega_{+}$ to be the disk enclosed by $\Sigma_R$ and $\Omega_{-}$ to be its complement, namely,
\[
	\Omega_\pm := \{(x,y) \in \R^2 \,|\, \pm (x^2 + y^2 - R^2) < 0\}\,.
\]

Of course, in this context it is convenient to pass to polar coordinates. We set
\[
\begin{cases}
r=\sqrt{x^2+y^2}\,, \\
\theta = \arctan (y/x)\,,
\end{cases}
\qquad \iff \qquad
\begin{cases}
x=r\cos\theta\,, \\
y=r\sin\theta\,.
\end{cases}
\]
Considering that
\[
\partial_x= \cos\theta\, \partial_r - \sin\theta\,\tfrac{1}{r}\,\partial_\theta\,, \qquad 
\partial_y= \sin\theta \,\partial_r + \cos\theta\,\tfrac{1}{r}\,\partial_\theta\,,
\]
and recalling \eqref{eq:paulim}, we obtain
\[
\siv \cdot \nabla = \sigma_r \partial_r + \sigma_\theta \,\tfrac{1}{r}\,\partial_\theta\,,
\]
where
\[
\sigma_r:=\begin{pmatrix}0 & e^{-i\theta} \\ e^{i\theta} & 0 \end{pmatrix} ,\qquad
\sigma_\theta:=\begin{pmatrix}0 & -i e^{-i\theta} \\ i e^{i\theta} & 0 \end{pmatrix} .
\]
Combining the above identities, we can rewrite the differential expression of the free Dirac operator in polar coordinates as 
\begin{equation}\label{eq:Dpolar}
-i \siv \cdot \nabla + m\,\sigma_3 = \begin{pmatrix} m & e^{-i\theta} \left(-i\partial_r-\frac{1}{r}\,\partial_\theta \right) \\  e^{i\theta} \left(-i \partial_r + \frac{1}{r}\,\partial_\theta \right) & -m \end{pmatrix} .
\end{equation}

Next, taking into account that the Hamiltonian operator under analysis commutes with the total angular momentum operator
\begin{equation}\label{eq:L3}
 J_3 
 = \big[ x(-i \partial_y) - y(-i\partial_x) \big] \sigma_0 + \tfrac{1}{2}\, \sigma_3
 = -i\, \sigma_0\, \partial_\theta + \tfrac{1}{2}\, \sigma_3\,,
\end{equation}
we proceed to write a generic $L^2$-spinor on $\R^2$ as
\begin{equation}\label{eq:polarspinor}
\psi(r,\theta) = \tfrac{1}{\sqrt{2\pi}}\sum_{k\in\mathbb Z}\begin{pmatrix} u_k(r)\, e^{i k \theta} \\ v_k(r)\, e^{i(k+1)\theta} \end{pmatrix} .
\end{equation}
Here and in the sequel, we are referring to the Hilbert-space isomorphism
\[
L^2(\R^2\,;\, \CC^2)\simeq \bigoplus_{k\in\Z} \left[ L^2(\R_+,r\,dr)\otimes h_k(\mathbb S^1)\right] ,
\]
where, for any $k \in \Z$, $h_k(\mathbb S^1) = \{c\, (e^{ik\theta}, e^{i(k+1)\theta})\,|\, c\in\CC\}$ is the one-dimensional subspace of $L^2(\mathbb S^1) \otimes \CC^2$ spanned by the eigenfunction of $J_3$ with eigenvalue $k + 1/2$.

Considering the unit normal vector $\nv = (\cos\theta,\sin\theta)$ pointing to the exterior of the disk $\Omega_+$, we further obtain
\[
\nv \cdot \siv  = \sigma_r\, .
\]
Referring to the trivial identification
\[
	L^2(\R_+,rdr) \simeq L^2((0,R),rdr) \oplus L^2((R,+\infty),rdr)\,, \qquad u \simeq u^{+} \oplus u^{-} ,
\]
the boundary condition in \eqref{eq:Ddom} reduces to
\begin{equation}\label{eq:bcradial}
\begin{cases}
-i\,[v_k]_R = (\eta+\tau)\, \langle u_k \rangle_R \,, \\
-i\,[u_k]_R = (\eta-\tau)\, \langle v_k \rangle_R \,,
\end{cases} \qquad
\mbox{for all } k \in \Z\,,
\end{equation}
where we set
\[
[u_k]_R := u_{k}^{+}(R^-) - u_{k}^{-}(R^+)\,, \qquad
\langle u_k \rangle_R := \tfrac{1}{2} \big[ u_{k}^{+}(R^-) + u_{k}^{-}(R^+) \big]\,.
\]
Combining the above observations, it is not hard to see that in the case under analysis the Hamiltonian \eqref{eq:Ddom}, \eqref{eq:Daction} can be decomposed as
\begin{equation}\label{eq:csum}
\DDet \sim \bigoplus_{k\in\mathbb Z}\, \DDet[k]\,,
\end{equation}
where the fiber operators are defined as follows:
\begin{eqnarray}
& \hspace{-2.7cm} \dom\big(\DDet[k]\big):=\Big\{\psih = \!\begin{pmatrix} u_{k}^{+} \vspace{0.1cm}\\ v_{k}^{+} \end{pmatrix} \oplus \begin{pmatrix} u_{k}^{-} \vspace{0.1cm}\\ v_{k}^{-} \end{pmatrix} \in L^2((0,R),rdr;\CC^2)\oplus L^2\big((R,+\infty),rdr;\CC^2\big)\quad \mbox{s.t.} \nonumber \\
& \hspace{0.cm}  (\partial_r - \tfrac{k}{r})u_{k}^{+},(\partial_r + \tfrac{k+1}{r})v_{k}^{+} \in L^2\big((0,R),rdr\big)\,,\; (\partial_r - \tfrac{k}{r})u_{k}^{-},(\partial_r + \tfrac{k+1}{r})v_{k}^{-} \in L^2((R,+\infty),rdr)\,, \nonumber \\
&  \hspace{12.cm} \mbox{and\, \eqref{eq:bcradial} holds}\,\Big\} ; \\
& \DDet[k]= \begin{pmatrix} m &  -i \left(\partial_r+\frac{k+1}{r} \right) \\ -i\left(\partial_r - \frac{k}{r}\right) & -m \end{pmatrix} . \label{eq:polarfiber}
\end{eqnarray}
\smallskip

We already mentioned that general results on the full Hamiltonian operator $\DDet$ defined in \eqref{eq:Ddom}, \eqref{eq:Daction} can be found in \cite{BHOP20}, for the case of a simple and smooth closed curve $\Sigma\subseteq \R^2$, with arbitrary parameters $\eta,\tau\in\R$. Here we limit ourselves to recall the following result, which is of interest for our purposes \cite[Thm. 1.2]{BHOP20}.

	\begin{proposition}
		Let $\Sigma$ be a smooth simple closed curve in the plane and $\eta,\tau \in \R$, with $\eta^2-\tau^2=4$. Then the Hamiltonian operator $\DDet$ defined in \eqref{eq:Ddom}, \eqref{eq:Daction} fulfills
		\begin{equation}
			\sigma_{\mathrm{ess}}(\DDet)=(-\infty,m]\cup \Big\{-\frac{\tau}{\eta}\,m\Big\}\cup[m,+\infty)\,.		
		\end{equation}
		Accordingly, 
		\begin{equation}
			\sigma_{\mathrm{disc}}(\DDet)\subseteq (-m,m)\,\setminus\, \Big\{-\frac{\tau}{\eta}\,m\Big\}\, .
		\end{equation}
	\end{proposition}

We are now in the position to state our main result.

	\begin{theorem}\label{thm:maincirc}
		Let $\Sigma=\Sigma_R$ for a fixed $R>0$, and let $\eta,\tau\in\R$ satisfy $\eta^2-\tau^2=4$. Then, the following statements hold.\\
		i) The isolated spectral point
			\begin{equation}
				\zet=-\frac{\tau}{\eta}\, m\in\sigma_{\mathrm{ess}}(\DDet)\,,
			\end{equation}
			and the threshold values $z_{\pm} = \pm m\in\sigma_{\mathrm{ess}}(\DDet)$ are \emph{not} eigenvalues of $\DDet$. 
		\\
		ii) The discrete spectrum of $\DDet$ contains a doubly infinite sequence of eigenvalues
			\begin{equation}\label{eq:zkeig}
				\{z_k\}_{k \in \Z} \subset (-m,m)\,,
			\end{equation}
			which are implicitly characterized as the solutions of the equation
			\begin{equation}\label{eq:evid}
				\frac{(\eta-\tau)(m- z)\,I_{k+1}(\sqrt{m^2-z^2}\,R)\,K_{k+1}(\sqrt{m^2-z^2}\,R)}{(\eta+\tau)(m+ z)\,I_{k}(\sqrt{m^2-z^2}\,R)\,K_{k}(\sqrt{m^2-z^2}\,R)}=1\,,
			\end{equation}
			where $I_\nu,K_\nu$ are the modified Bessel functions of the first and second kind, respectively.\\
			These eigenvalues $\{z_k\}_{k \in \Z}$ accumulate at $\zet$, that is,
			\begin{equation}\label{eq:zacc}
				\lim_{k\to\pm\infty} z_k= \zet\,.
			\end{equation}
		iii) The eigenfunctions associated to the eigenvalues $\{z_k\}_{k \in \Z}$ are given by
			\begin{equation}\label{eq:formk}
				\psi_k(r,\theta)= \tfrac{1}{\sqrt{2\pi}} \begin{pmatrix} u_{k}(r)\, e^{ik\theta} \\ v_{k}(r)\, e^{i(k+1)\theta}\end{pmatrix} ,
			\end{equation}
			in polar coordinates, with
			\begin{equation}\label{eq:uvIK}
				\begin{cases}
					u_k(r) = a_k \left[\one_{(0,R)}(r)\, I_k(t) + \one_{(R,+\infty)}(r)\, c_{k}\, K_k(t)\right]_{t\,=\,\sqrt{m^2-z_k^2}\,r}\, ,\vspace{0.1cm} \\
					v_k(r) = -i\,a_k\,\sqrt{\frac{m-z_k}{m+z_k}} \left[\one_{(0,R)}(r)\, \, I_{k+1}(t) - \one_{(R,+\infty)}(r)\, c_{k}\, K_{k+1}(t) \right]_{t\,=\,\sqrt{m^2-z_k^2}\,r} \,,
				\end{cases}
			\end{equation}
			where
			\begin{equation}\label{eq:const}
				c_{k} := \left[\frac{(\eta-\tau)\sqrt{m-z_k}\,I_{k+1}(t_k) - 2\sqrt{m+z_k}\, I_k(t_k)}{(\eta-\tau)\sqrt{m-z_k}\,K_{k+1}(t_k)-2\sqrt{m+z_k}\,K_k(t_k)}\right]_{t_k\,=\,\sqrt{m^2-z_k^2}\,R} ,
			\end{equation}
			and $a_k\in \CC$ is a normalization constant. In particular, $\|\psi_k\| = 1$ if
			\begin{align}\label{eq:akunit}
				a_k
					& = \frac{\sqrt{2}}{R}\, \bigg[
I_k^2(t_k) - I_{k-1}(t_k)I_{k+1}(t_k)
+ |c_{k}|^2 \Big(K_{k-1}(t_k) K_{k+1}(t_k) - K^2_k(t_k) \Big)
				 \\
				& \hspace{0.5cm} + \frac{m-z_k}{m+z_k} \left[ I_{k+1}^2(t_k) - I_{k}(t_k R)I_{k+2}(t_k)
+ |c_{k}|^2 \Big(K_{k}(t_k)K_{k+2}(t_k) - K_{k+1}^2(t_k)\Big)\right]
 \bigg]_{t_k \,=\,\sqrt{m^2-z_k^2}\,R}^{-1/2} . \nonumber
			\end{align}
	\end{theorem}
	
	\begin{remark}
			Our arguments do not suffice to establish in full generality whether the discrete spectrum of $\DDet$ actually coincides with the set of solutions to \eqref{eq:evid}. As a matter of fact, claim (ii) singles out only a subsequence $\{z_k\}_{k \in \Z}$ of that set, and two distinct issues could prevent an exact identification: in principle, there may exist multiple distinct solutions to \eqref{eq:evid} associated with the same $k \in \Z$, and some of the eigenvalues could have finite multiplicity greater than one. Simplicity of the eigenvalues can be guaranteed, via \cref{prop:fg}, only when $\eta$ and $\tau$ have opposite signs. Further partial evidence is provided by \cref{cor:asef} below, which shows that, at least for sufficiently large $k$, the solutions $z_k$ to \eqref{eq:evid} are simple and pairwise distinct.
			
	\end{remark}

	\begin{remark}\label{rmk:nom}
		As a direct consequence of \cref{thm:maincirc}, the thresholds of the essential spectrum, namely $\pm m\in\sigma_{\mathrm{ess}}(\DDet)$, are not accumulation points for eigenvalues in the gap.
	\end{remark}

	\begin{figure}[t!]
        \centering
        \includegraphics[scale=.5]{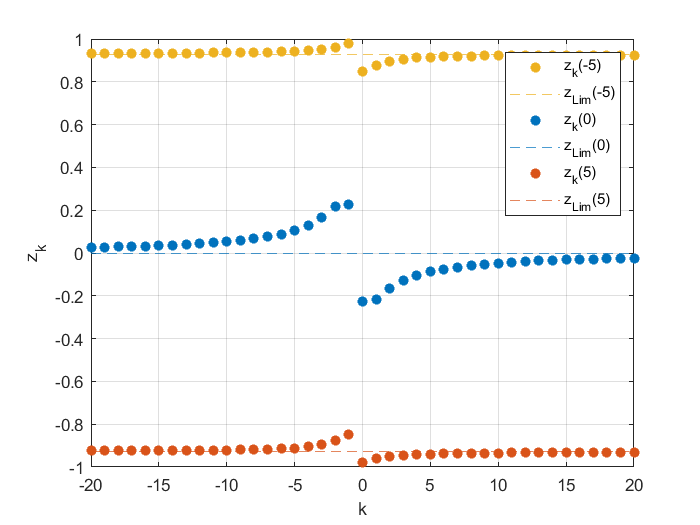}
		\caption{The discrete eigenvalues $z_k$ of $\DDet$, for $\eta = +\sqrt{4+\tau^2}$ and $\tau=-5,0,5$, in yellow, blue and red, respectively.}
		\label{fig:ev}
   \end{figure}

	\begin{remark}
		\cref{thm:mainline} and \cref{thm:maincirc} reveal a qualitative difference between the spectral behavior of critical Dirac delta-shell interactions supported a straight line and on a circle. At first sight, this distinction is somewhat unexpected. Indeed, a natural line of reasoning would attribute the infinite degeneracy of the eigenvalue $\zet$ in the straight-line case to the translational symmetry of that configuration. By analogy, one might therefore expect $\zet$ to remain a degenerate eigenvalue when the interaction is supported on a circle, now with rotations playing the role of translations.
		Our results show that this intuition is incorrect. Our expectation is that the straight-line configuration is exceptional, whereas the spectral behavior observed for the circle is generic. This leads us to formulate the following conjecture.	
		\medskip
		
		\textbf{\textit{Conjecture 1.}} For critical delta-shell interactions satisfying $\eta^2-\tau^2 = 4$ and supported on a generic smooth, simple (not necessarily closed) curve, the isolated point $\zet$ in the essential spectrum is an accumulation point of eigenvalues, but it is not itself an eigenvalue. 
		\medskip 
		
		\noindent
		Concerning the above conjecture, we mention that in \cite{BEHT25} the persistence of the isolated point $\zet$ in the essential spectrum is proved in the critical case $\eta^2-\tau^2 = 4$, considering unbounded curves with straight ends. However, its nature does not seem to be investigated. It would be interesting to understand whether local deformations of a straight line may break the degeneracy.
		\smallskip 

		\noindent 
		We refer to the forthcoming \cref{cor:asef} for results that partially support Conjecture 1. In particular, it is shown therein that the asymptotic expansion of the eigenvalues $z_k$ for large $k$ depends explicitly on the radius (and hence on the curvature) of the circle.
	\end{remark}

	\begin{remark}\label{rmk:inf}
		\cref{thm:maincirc} establishes the existence of an infinite number of eigenvalues for critical strengths $\eta,\tau \in \R$ with $\eta^2-\tau^2 = 4$, in the case of delta-shell interactions supported on a circle. On the other hand, it is known that for generic closed simple curves the discrete spectrum has finite cardinality whenever $\eta^2-\tau^2 \neq 4$, see \emph{e.g.}, \cite[Thm. 1.1]{BHOP20}. In light of this, it is natural to formulate the following conjecture.
		\medskip
		
		\textbf{\textit{Conjecture 2.}} The number of discrete eigenvalues of $\DDet$ diverges as the critical condition $\eta^2-\tau^2 = 4$ is approached, that is,
			\begin{equation}\label{eq:discinf}
				\lim_{\eta \to \pm \sqrt{4+\tau^2}} \# \big[\sigma_{\mathrm{disc}}(\DDet)\big] = +\infty\,, \qquad \mbox{for any $\tau \in \R$}\,.
			\end{equation}
		\smallskip 
		
		\noindent
		We expect this behavior to hold for interactions supported on any closed simple curve, and not only for the circular geometry considered here. Of course, one could go further and examine the rate at which infinity cardinality is reached in \eqref{eq:discinf}.
	\end{remark}

	\begin{remark}
		The spinors $\psi_k$	characterized in item (iii) of \cref{thm:maincirc} are simultaneous eigenfunctions of the Dirac Hamiltonian $\DDet$ and of the total angular momentum $J_3$, see \eqref{eq:L3}, with
		\[
			J_3 \psi_k = (k+1/2)\,\psi_k\,, \qquad \mbox{for any $k \in \Z$}\,.
		\]	
		One may also derive explicit expressions for the expectation values of the spin component $\sigma_3$, the radial coordinate $r$	and the velocity components 
		\[
			v_r = i[\DDz,r] = \sigma_r\,, \qquad 
			v_\theta = i[\DDz,\theta] = \frac{1}{r}\,\sigma_\theta\,.
		\]
		As a matter of fact, a simple algebraic cancellation yields
		\begin{equation}\label{eq:vr}
			\langle v_r \rangle_k := \langle \psi_k, v_r\, \psi_k \rangle = 0\,, \qquad \mbox{for all $k \in \Z$}\,,
		\end{equation}
		and hence
		\begin{align*}
			\langle (\Delta v_r)^2 \rangle_k := \|v_r \psi_k\|^2 - \langle \psi_k, v_r \psi_k \rangle^2 = \|\psi_k\|^2 = 1\,.
		\end{align*}
		Eq. \eqref{eq:vr} indicates that, as expected, the bound states $\psi_k$ exhibit no net motion in the radial direction.
		\\
		The expectation values of the spin and of the angular velocity are given by
		\begin{multline*}
			\langle \sigma_3 \rangle_k :=
			\langle \psi_k, \sigma_3\, \psi_k \rangle 
			= \frac{|a_k|^2}{m^2 - z_k^2} \left[
			\int_{0}^{\sqrt{m^2-z_k^2}\,R} dt\,t \left(I_k^2(t) - \frac{m-z_k}{m+z_k}\,I_{k+1}^2(t) \right) \right. \\
			\left. -\, c_k^2 \int_{\sqrt{m^2-z_k^2}\,R}^{+\infty} dt\,t \left(K_k^2(t) - \frac{m-z_k}{m+z_k}\,K_{k+1}^2(t) \right)
			\right] ,
		\end{multline*}
		\begin{equation*}
			\langle v_\theta \rangle_k :=
			\langle \psi_k, v_\theta\, \psi_k \rangle 
			= \frac{2\,|a_k|^2}{m + z_k} \left[c_{k}^2 \int_{\sqrt{m^2-z_k^2}\,R}^{\infty} K_{k}(t) K_{k+1}(t)\,dt - \int_{0}^{\sqrt{m^2-z_k^2}\,R} I_{k}(t) I_{k+1}(t)\,dt \right].
		\end{equation*}
		Although the integrals appearing above can be expressed in terms of generalized hypergeometric functions and Meijer $G$-functions, see \cite[Ch. 16]{NIST}, the resulting formulas are rather cumbersome and offer limited additional insight. Instead, let us highlight some noteworthy qualitative features, that follow from basic properties of Bessel functions.	
		\\
		In general, as no cancellations occur for generic choices of the parameters, we have
		\begin{equation*}
			\langle \sigma_3 \rangle_k \neq 0\,, 
			\qquad \mbox{and}\qquad 
			\langle v_\theta\rangle_k \neq 0\,.
		\end{equation*}
		To say more, for fixed $k \in \Z$, it is possible to prove that
		\begin{gather*}
			\langle \sigma_3 \rangle_k =  -\,\frac{\tau}{\eta} + O\left(\tfrac{1}{R^2}\right), \qquad \mbox{and}\qquad
			\langle v_\theta\rangle_k = -\,\frac{1}{|\eta| R}\,e^{-\frac{4 m R}{|\eta|}} \Big[1 + O\left(\tfrac{1}{R}\right) \Big]\,,
			\qquad \mbox{as $R \to +\infty$}\,.
		\end{gather*}
		The asymptotic behavior of $\langle \sigma_3 \rangle_k$ agrees with that obtained in the straight line case, cf. \cref{cor:eigline}, as expected. By contrast, it is perhaps surprising that $\langle v_\theta\rangle_k$ becomes negative definite (although exponentially small) for large radii, regardless of the other parameters.
	\end{remark}	
	
	\begin{remark}
		The accumulation law in \eqref{eq:zacc} shows that energy levels associated with large total angular momentum cluster in the infrared sector of $\DDet$. This behavior is counterintuitive from a physical perspective: semiclassical reasoning would suggest that large angular momentum corresponds to high kinetic energy, and hence to eigenvalues drifting toward the ultraviolet sector, or at least to the continuous thresholds. Here, instead, high–angular–momentum eigenstates remain strongly localized near the interaction support, as if the centrifugal growth were exactly compensated by the singular coupling, effectively pinning their energies near a resonance value. In this connection, see also the forthcoming \cref{rem:concen} and the related Figs. \ref{fig:modu}-\ref{fig:l2t_2b}.
		 This delicate cancellation, together with the resulting phenomenology, should perhaps be viewed as an additional signature of criticality.
	\end{remark}
	\medskip

The results in \cref{thm:maincirc} are complemented by some further observations.
\smallskip

\noindent
We first analyze the asymptotic behavior of the eigenvalues $\{z_k\}_{k \in \Z}$ and eigenfunctions $\{\psi_k\}_{k \in \Z}$ from \cref{thm:maincirc} as $k \to \pm \infty$.

	\begin{corollary}\label{cor:asef}
		Under the assumptions of \cref{thm:maincirc}, for $k\to \pm \infty$ there holds
			\begin{equation}\label{eq:evas}
				z_k/m = -\frac{\tau}{\eta} - \left(\frac{2}{\eta^2}\right) \frac{1}{k} + \left(\frac{\eta+\tau}{\eta^3}\right)\frac{1}{k^2}+ \left(\frac{(8 m R)^2 - (\eta + \tau)^2}{2\eta^4}\right)\frac{1}{k^3} +O\!\left(\frac{1}{k^4}\right) .
			\end{equation}
		Moreover, 
			\begin{equation}\label{eq:zerolim}
				\lim_{k \to \pm \infty} \psi_k(r)=0\,,\qquad \mbox{for all\, $r\in [0,+\infty) \,\setminus\, \{R\}$}\,.
			\end{equation}
	\end{corollary}

	\begin{remark}\label{rmk:universal}
		In the asymptotic expansion \eqref{eq:evas} one sees that the first three leading terms \emph{do not} depend on the radius $R>0$ of the circle, whereas the dependence on $R$ enters into play at the fourth-order. Incidentally, one may derive explicit expressions also for higher order terms in the expansion \eqref{eq:evas}.
		This observation suggests that, for a generic closed simple curve $\Sigma\subseteq \R^2$, curvature effects may appear only at order $O(k^{-3})$, while the first three terms of the expansion could be universal. From a broader perspective, it is natural to conjecture that the leading terms in the accumulation law are determined by topological invariants.
		\medskip
		
		\textbf{\textit{Conjecture 3.}} The accumulation law for the eigenvalues $z_k$ is governed by topological invariants of the geometric locus where the interaction is supported.
		\medskip 

		\noindent
		This behavior would be reminiscent of the Efimov effect, in which a three-body quantum system with pairwise interactions tuned to resonance exhibits an infinite sequence of eigenvalues accumulating at the continuous threshold according to a universal law that is independent of the specific form of the two-body interaction potentials \cite{FFT23,NE17}.
	\end{remark}  
	
	\begin{remark}\label{rem:concen}
		Taking into account that all the eigenstates $\psi_k$ have unit norm in $L^2(\R^2;\CC^2)$, see item (iii) in \cref{thm:maincirc}, it follows from \eqref{eq:zerolim} that they cannot converge in strong $L^2$-sense. Indeed, the plots reported in Fig. \ref{fig:modu} show that the modulus squared of the eigenfunctions tends to concentrate at the interface $\{r = R\}$.
	\end{remark}

	\begin{figure}[t!]
        \centering
        \includegraphics[scale=.5]{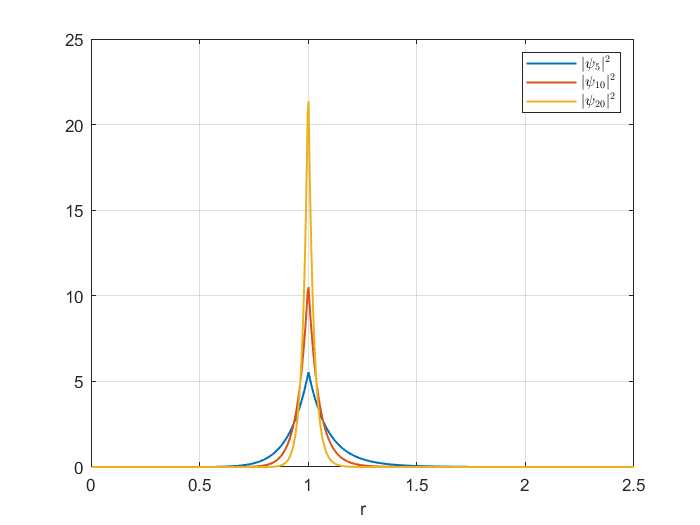}
		\caption{Plots of the square-modulus $\vert \psi_k\vert^2$ of the normalized eigenfunctions \eqref{eq:formk}, as a function of the radial coordinate $r > 0$, for $k=5$ (blue), $k=10$ (red) and $k=20$ (yellow), for $m = R = 1$ and $\tau = 0$.}
		\label{fig:modu}
	\end{figure}
	\medskip

The purely electrostatic case $\tau=0$ is even more exceptional, owing to specific symmetry properties of the Hamiltonian operator. In this setting, the discrete spectrum is symmetric around the accumulation point $\zet =0$. More precisely, the following holds.

	\begin{corollary}\label{cor:symev}
		Under the assumptions of \cref{thm:maincirc}, if $\tau=0$ the discrete eigenvalues of $\DD_{\eta,0}$ are symmetric with respect to $\zet=0$, meaning that
		\begin{equation}\label{eq:symz}
			z\in\sigma_{\mathrm{disc}}(\DD_{\eta,0}) \qquad \iff \qquad  -z\in\sigma_{\mathrm{disc}}(\DD_{\eta,0})\,.
		\end{equation}
	\end{corollary}
 
	\begin{remark} 
		Symmetries of the spectrum of $\DDet$ were previously identified in \cite[Prop. 4.15]{BHOP20} for the case of a generic, smooth closed simple curve $\Sigma\subseteq \R^2$. The key observation here is that the Hamiltonian operator $\DDet$ is mapped to its opposite by the combined, conjugate action of the following operators.
		\begin{enumerate}

			\item \emph{Charge conjugation}, that is the anti-unitary operator
				\begin{equation}\label{eq:cc}
					\Con: L^2(\R^2\,;\CC^2)\to L^2(\R^2\,;\CC^2)\,, \qquad \Con \psi := \sigma_1\overline{\psi}\,,
			\end{equation}  
				where $\overline{\psi}$ denotes the complex conjugate of $\psi$.
			
			\item \emph{Anti-symmetric phase rotation}, that is the unitary and self-adjoint operator
				\begin{eqnarray}
					& \UU : L^2(\Omega_+\,;\CC^2)\oplus L^2(\Omega_-\,;\CC^2) \to L^2(\Omega_+\,;\CC^2)\oplus L^2(\Omega_-\,;\CC^2)\,,
						\nonumber \\			
					& \UU(f_+ \oplus f_-) := f_+ \oplus (-f_-)\,. \label{eq:U}
				\end{eqnarray}
		\end{enumerate}
		The name proposed for $\UU$ is motivated by the fact that it leaves the wavefunction invariant in the exterior domain $\Omega_-$, while implementing a rotation by $\pi$ of the phase in the interior domain $\Omega_+$.
		
		As a matter of fact, with the above definitions one has
		\begin{equation}\label{eq:CUH}
			\Con\, \UU\, \DD_{\eta,0}\, (\Con\, \UU)^{-1} = - \DD_{\eta,0}\,.
		\end{equation}			
		Even though here we are considering only the case of a circle, the proof of \cref{cor:symev} actually works for a generic smooth closed simple curve, provided that $\tau=0$.
	\end{remark}
	
	\begin{figure}[t!]
	\begin{minipage}[b]{0.4\textwidth}
        \centering
        \includegraphics[width=\textwidth]{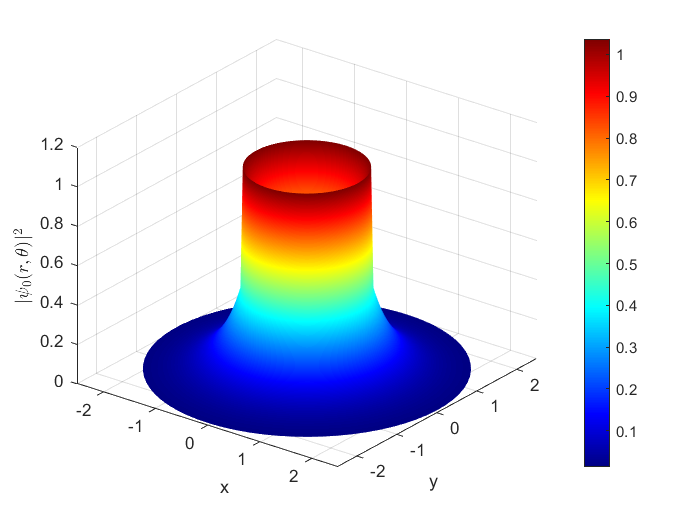}
    \end{minipage}
    \begin{minipage}[b]{0.4\textwidth}
        \centering
        \includegraphics[width=\textwidth]{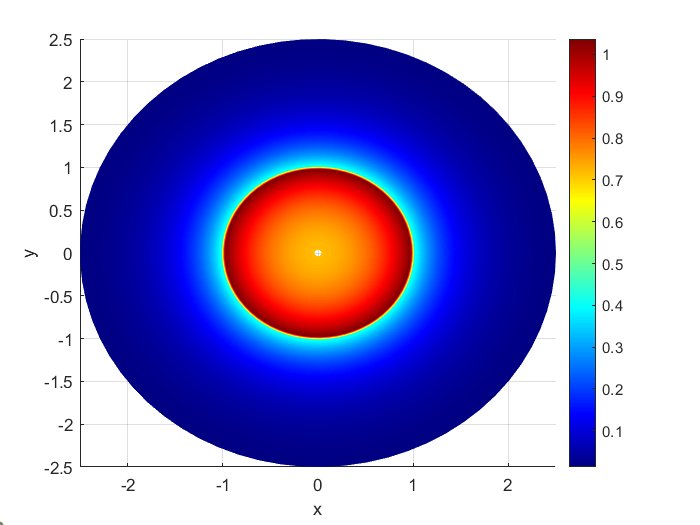}
    \end{minipage}
    \caption{Plot of $\vert \psi_k\vert^2$ for $k=0$ and $\eta=2\sqrt 2$, $\tau=-2$.}
    \label{fig:l2t_2a}
	\end{figure}	
	\begin{figure}[t!]
    \begin{minipage}[b]{0.4\textwidth}
        \centering
        \includegraphics[width=\textwidth]{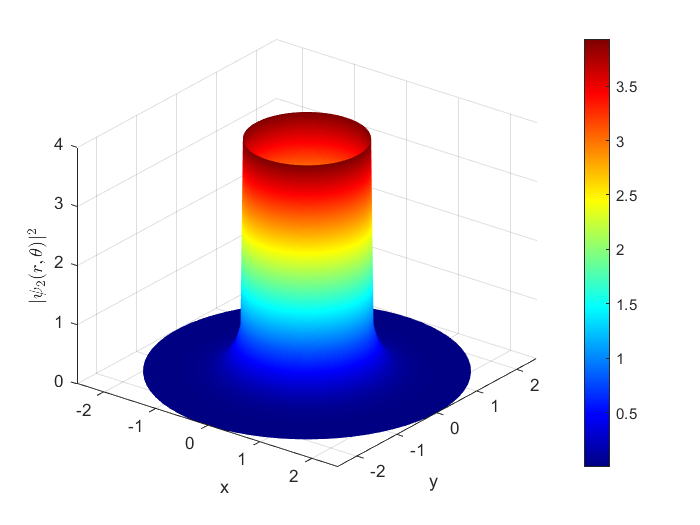}
    \end{minipage}
    \begin{minipage}[b]{0.4\textwidth}
        \centering
        \includegraphics[width=\textwidth]{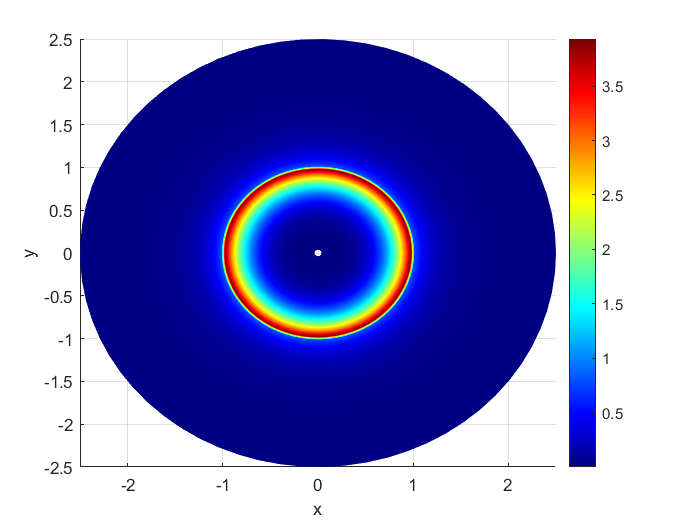}
    \end{minipage}
    \caption{Plot of $\vert \psi_k\vert^2$ for $k=2$ and $\eta=2\sqrt 2$, $\tau=-2$.}
    \label{fig:l2t_2b}
	\end{figure}

	\begin{remark}
		Numerical results suggest that the symmetry exhibited in \cref{cor:symev} is indeed broken when $\tau \neq 0$, see Fig. \ref{fig:ev}.
	\end{remark}
 
	\begin{remark}
		At first sight, the asymptotic expansion \eqref{eq:evas} seems to conflict with the spectral symmetry established in \cref{cor:symev}, at least for $\vert k\vert$ large. In fact, it appears to be incompatible with an identity of the form $z_k = -z_{-k}$. However, this apparent discrepancy is only a matter of labeling: the eigenvalues are not a priori indexed so as to satisfy this identity.
		When $\tau = 0$ it is easy to see that the characterizing equation \eqref{eq:evid} is invariant under the simultaneous replacements $z \to - z$ and $k \to - k-1$, so that the correct symmetry condition takes the form
		\[
			z_k = - z_{-k-1} \qquad \mbox{for all $k \in \Z$}\,.
		\]
	\end{remark}

\section{Proofs}

\subsection{Proof of the statements in \cref{sec:Dline}}
We begin by addressing the case of the line, taking $\Sigma = \Sigma_L$ as in \eqref{eq:line}. For completeness, we present in the forthcoming proof a self-contained derivation of the eigenfunctions of $\DDet$, reproducing the computation from scratch.

\begin{proof}[Proof of \cref{thm:mainline}]
We look for solutions $\psi \in \dom(\DDet)$ of the eigenvalue problem
\begin{equation}\label{eq:eigenDDet}
	\DDet \psi = z\, \psi \,,
\end{equation}
for a generic $z \in \R$. By partial Fourier transform, for any fixed $k \in \R$, we can refer to the fiber operator $\DDket$ and to the corresponding reduced eigenvalue equation
\begin{equation}\label{eq:eigenDDket}
	\DDket \psih = z\,\psih\,,
\end{equation}
where $\psih \equiv \psih(\,\cdot\,,k) = (\Fou_2 \psi)(\,\cdot\,,k) \in \dom(\DDket)$. In the sequel, we first determine all the distributional solutions of \eqref{eq:eigenDDket}, and then proceed to implement the square-integrability requirement and the boundary condition at $x = 0$.

Setting $\psih = \psih_+ \oplus \psih_-$, with $\psih_{\pm} = (\hat{u}_\pm,\hat{v}_\pm) \in L^2(\R_{\pm},dx; \CC^2)$, \eqref{eq:eigenDDket} reduces to the following system of first order ODEs on $\R_{\pm}$, respectively:
\begin{equation}\label{eq:eeini}
\begin{cases}
	-i(\partial_x + k) \hat{v}_\pm - (z-m) \hat{u}_\pm = 0\,, \\
	-i(\partial_x - k) \hat{u}_\pm - (z+m) \hat{v}_\pm = 0\,.
\end{cases}
\end{equation}
By substitution, we obtain
\begin{equation*}
\begin{cases}
	(- \partial_{xx} + k^2 + m^2 - z^2) \hat{u}_\pm = 0\,, \\
	(- \partial_{xx} + k^2 + m^2 - z^2) \hat{v}_\pm = 0\,,
\end{cases}
\end{equation*}
for which the generic solution is of the form
\begin{equation}\label{eq:gensol}
\begin{cases}
	\hat{u}_\pm(x) = a_\pm(k)\, e^{\sqrt{k^2 + m^2 - z^2}\,x} + b_\pm(k)\, e^{-\sqrt{k^2 + m^2 - z^2}\,x} \,, \\
	\hat{v}_\pm(x) = c_\pm(k)\, e^{\sqrt{k^2 + m^2 - z^2}\,x} + d_\pm(k)\, e^{-\sqrt{k^2 + m^2 - z^2}\,x} \,,
\end{cases}
\end{equation}
where $a_\pm(k),b_\pm(k),c_\pm(k),d_\pm(k)$ are complex coefficients, depending solely on the momentum variable $k \in \R$.
It is evident that, if $|z| \geqslant \sqrt{m^2 + k^2}$, none of the functions in \eqref{eq:gensol} belongs to $L^2(\R,dx)$. Therefore, any admissible eigenvalue of $\DDket$ must indeed lie in the (shifted) open gap interval, that is
\[
	\sigma_{\mathrm{p}}(\DDket) \subset \big(-\sqrt{m^2+k^2},\sqrt{m^2+k^2}\,\big)\,.
\]
Taking this into account, in order to ensure the square-integrability of the solutions \eqref{eq:gensol} for any given $z \in (-\sqrt{m^2+k^2},\sqrt{m^2+k^2}\,)$ and $x \to \pm \infty$, we must fix
\begin{equation}
	a_{+} = b_{-} = c_{+} = d_{-} = 0\,.
\end{equation}
Finally, notice that the initial equations \eqref{eq:eeini} require
\begin{equation}\label{eq:bcgen}
\begin{cases}
	i (\sqrt{k^2 + m^2 - z^2} + k) b_+(k) - (z+m) d_+(k) = 0\,, \\
	-i (\sqrt{k^2 + m^2 - z^2} - k) a_-(k) - (z+m) c_-(k) = 0\,.
\end{cases}
\end{equation}

The above arguments prove that any square-integrable solution of \eqref{eq:eigenDDket} is given by
\begin{equation}\label{eq:algsol}
	\psih(x,k) = \left[
					\theta(-x)\,a_-(k) \begin{pmatrix} 1 \\ -i\, \frac{\sqrt{k^2 + m^2 - z^2} - k}{m+z} \end{pmatrix}
					+ \theta(x)\,b_+(k) \begin{pmatrix} 1 \\ i\,\frac{\sqrt{k^2 + m^2 - z^2} + k}{m+z} \end{pmatrix}
				\right] e^{-\sqrt{k^2 + m^2 - z^2}\,\vert x\vert} \,.
\end{equation}
For any function of this form, the boundary condition \eqref{eq:bcLam}, with $\Lambda_{\eta,\tau}$ as in \eqref{eq:Lamcrit}, entails
\begin{align}\label{eq:eigenmat}
	\begin{pmatrix} \frac{\eta - \tau}{2}\, \frac{\sqrt{k^2 + m^2 - z^2} - k}{m+z} & 1 \\ \frac{i(\eta + \tau)}{2} & i\,\frac{\sqrt{k^2 + m^2 - z^2} + k}{m+z} \end{pmatrix}  \begin{pmatrix} a_-(k) \\ b_+(k) \end{pmatrix}= 0\,.
\end{align}
Of course, this equation possesses a non-trivial solution if and only if the determinant of the $2\times 2$ matrix on the left-hand side vanishes, which is ultimately equivalent to
\[ 
\eta z + \tau m = 0\,.
\]
The related solution is obtained by fixing
\[
b_{+}(k) = -\frac{\eta - \tau}{2}\, \frac{\sqrt{k^2 + m^2 - z^2} - k}{m+z}\,a_-(k)\,,
\]
which ensures, in particular, that
\[
b_+(k) \begin{pmatrix} 1 \\ i\,\frac{\sqrt{k^2 + m^2 - z^2} + k}{m+z} \end{pmatrix}
= a_{-}(k)\, \Lambda_{\eta,\tau} \begin{pmatrix} 1 \\ -i\, \frac{\sqrt{k^2 + m^2 - z^2} - k}{m+z} \end{pmatrix}.
\]

Summing up, the above arguments prove that $\zet = -\tau m/\eta$ is the only admissible eigenvalue of $\DDet$ and that the partial Fourier transform of the corresponding eigenfunctions are given by
\begin{align*}
	\psih(x,k) = \big[\theta(-x)\,\sigma_0 + \theta(x)\,\Lambda_{\eta,\tau} \big] \begin{pmatrix} 1 \\ -i\, \frac{\sqrt{k^2 + m^2 - z^2} - k}{m+z} \end{pmatrix} a_-(k)\, e^{-\sqrt{k^2 + m^2 - z^2}\,\vert x\vert}\,.
\end{align*}
The expression \eqref{eq:eignfunexp} reported in the statement of \cref{thm:mainline} is obtained by applying the inverse, partial Fourier transform $\Fou_2^{-1}$ to any such function $\psih$, noting that $(\sqrt{k^2 + m^2 - z^2} - k)(\sqrt{k^2 + m^2 - z^2} + k) = (m+z)(m-z) = m^2 (1- \tau^2/\eta^2) = 4 m^2/\eta^2$ and fixing the normalization constant as
\[
	a_{-}(k) := \sqrt{\tfrac{1}{2}\,(1 - \tfrac{\tau}{\eta})(\sqrt{k^2 + 4 m^2/\eta^2} + k)}\,\Nh(k) \,.
\]
The identity \eqref{eq:normeq} can be easily derived by Plancherel theorem and some elementary computations.
\end{proof}

The results reported in \cref{cor:eigline} follow from Plancherel's theorem and straightforward, albeit rather lengthy, calculations. We omit the details as they offer no significant insight. We only notice the elementary identity $y = \Fou_2^{-1} (i \partial_k) \Fou_2$, which yields, in particular,
\[
	\langle y \rangle_{\eta,\tau} = \langle \Fou_2 \psili, ( i \partial_k) \Fou_2 \psili\rangle\,, \qquad
	\|y \,\psili\| = \| i \partial_k \Fou_2 \psili \|\,.
\]
By similar computations, it can be easily verified that $\psili$ is \emph{not} an eigenfunction of any observable considered in \cref{cor:eigline}.
\medskip

\subsection{Proof of the statements in \cref{sec:Dcircle}}
We now proceed to discuss the case of the circle, considering $\Sigma = \Sigma_R$ as in \eqref{eq:circle}. For the ease of presentation, we split the proof of \cref{thm:maincirc} in various steps.

\begin{lemma}\label{lem:implicit}
Under the assumptions of \cref{thm:maincirc}, the eigenvalues of $\DDet$ are determined by the implicit relation \eqref{eq:evid} and the associated eigenfunctions have the form written in \eqref{eq:formk}, \eqref{eq:uvIK}, \eqref{eq:const} and \eqref{eq:akunit}.
\end{lemma}
\begin{proof}[Proof of Lemma \ref{lem:implicit}]
We look for non-trivial solutions $\psi\in\dom(\DDet)$ of the eigenvalue equation
\begin{equation}\label{eq:eveq}
    \DDet\psi=z\psi\,.
\end{equation}
Exploiting the angular momentum decomposition \eqref{eq:polarspinor}, we henceforth work at fixed angular momentum $k\in\mathbb{Z}$, considering \eqref{eq:eveq} for the corresponding fiber operator $\DDet[k]$, namely,
\begin{equation}\label{eq:eveqk}
\DDet[k]\psi_k= z\,\psi_k\,,\qquad \psi_k\in\dom(\DDet[k])\,.
\end{equation}
 
{\sl 1) Algebraic solutions.}
We first solve \eqref{eq:eveqk}, by seeking solutions $\psi_k$ of the form \eqref{eq:formk}. This amounts to performing a discrete Fourier decomposition in the angular variable for each spinor component, independently, and involves no loss of generality. With this ansatz, \eqref{eq:eveqk} reduces to the following system of first-order ODEs for the couples of radial functions $(u^{\pm}_k,v^{\pm}_k)$, to be solved separately in the regions $0<r<R$ and $r>R$:
\begin{equation}\label{eq:ode}
    \begin{cases}
       i \partial_r v^{\pm}_k+i\,\frac{k+1}{r}\,v^{\pm}_k-(m-z)u^{\pm}_k = 0\,, \\
       i \partial_r u^{\pm}_k-i\,\frac{k}{r}\,u^{\pm}_k+(m+z)v^{\pm}_k = 0\,,
    \end{cases}
    \quad \iff \qquad
    \begin{cases}
    u^{\pm}_k=\frac{i}{m-z}\left(\partial_r +\frac{k+1}{r} \right)v^{\pm}_k\,, \\
    v^{\pm}_k=\frac{i}{m+z}\left(-\partial_r +\frac{k}{r} \right) u^{\pm}_k\,.
    \end{cases}
\end{equation}
A simple substitution produces two modified Bessel equations, namely,
\[
\begin{cases}
r^2 \partial_{rr} u^{\pm}_k +r \partial_r u^{\pm}_k-(k^2+(m^2-z^2)r^2)u^{\pm}_k=0\,, \\
r^2 \partial_{rr} v^{\pm}_k +r \partial_r v^{\pm}_k -((k+1)^2+(m^2-z^2)r^2)v^{\pm}_k=0\,.
\end{cases}
\]
The general solutions of these equations are of the form, see \emph{e.g.} \cite[\S 10.25]{NIST},
\begin{equation}\label{eq:gensolc}
\begin{cases}
u^{\pm}_k(r)=a^{\pm}_k\, I_k(\sqrt{m^2-z^2}\,r)+b^{\pm}_k\, K_k(\sqrt{m^2-z^2}\,r)\,, \\
v^{\pm}_k(r)=c^{\pm}_k\, I_{k+1}(\sqrt{m^2-z^2}\,r)+ d^{\pm}_k\, K_{k+1}(\sqrt{m^2-z^2}\,r)\,,
\end{cases}
\end{equation}
where $I_\nu,K_\nu$ are the modified Bessel functions of second kind and order $\nu$, while $a^{\pm}_k,b^{\pm}_k,c^{\pm}_k,d^{\pm}_k\in\CC$ are constant coefficients. Using the recurrence relations \cite[Eq. 10.29.2]{NIST}
\[
I'_\nu=I_{\nu\pm1}\pm\frac{\nu}{r}\,I_\nu\,,\qquad K'_\nu=-K_{\nu\pm1}\pm\frac{\nu}{r}\,K_\nu\,,
\]
it can be checked that the original equations in \eqref{eq:ode} are fulfilled if and only
\begin{equation}
    \begin{cases}
       i \sqrt{m^2-z^2}\, [c^{\pm}_k I_{k} - d^{\pm}_kK_{k}]-(m-z)[a^{\pm}_k I_k + b^{\pm}_k\, K_k] = 0\,, \\
       i \sqrt{m^2-z^2}\, [a^{\pm}_k I_{k + 1} - b^{\pm}_k K_{k+1}]+(m+z)[c^{\pm}_k I_{k+1}+ d^{\pm}_k K_{k+1}] = 0\,,
    \end{cases}
\end{equation}
where we omitted the arguments of the Bessel functions for brevity. To satisfy the above relations, we must fix
\[
	c^{\pm}_k = - i\,\sqrt{\frac{m-z}{m+z}}\,a^{\pm}_k\,, \qquad 
	d^{\pm}_k = i\,\sqrt{\frac{m-z}{m+z}}\,b^{\pm}_k\,.
\]
We thus obtain
\begin{equation}\label{eq:uv}
\begin{cases}
u^{\pm}_k(r)=a^{\pm}_k\, I_k(\sqrt{m^2-z^2}\,r)+b^{\pm}_k\, K_k(\sqrt{m^2-z^2}\,r)\,, \\
v^{\pm}_k(r)=- i\,\sqrt{\frac{m-z}{m+z}}\,a^{\pm}_k\, I_{k+1}(\sqrt{m^2-z^2}\,r)+ i\,\sqrt{\frac{m-z}{m+z}}\,b^{\pm}_k\, K_{k+1}(\sqrt{m^2-z^2}\,r)\,.
\end{cases}
\end{equation}

{\sl 2) Ensuring square-integrability.}
Let us now examine the square-integrability of the solutions \eqref{eq:uv}, considering $u_+,v_+$ for $0<r<R$ and $u_-,v_-$ for $r>R$, separately. On one side, let us recall the small argument asymptotics
\begin{equation}\label{eq:sasym}
I_\nu(t)\sim \frac{1}{\Gamma(\nu+1)}\left(\frac{t}{2} \right)^{\!\vert \nu\vert} , \qquad
K_\nu(t)\sim\begin{cases}-\ln(t/2)-\gamma\,, & \mbox{if }\nu=0 \\ \frac{\Gamma(\nu)}{2}\left(\frac{t}{2} \right)^{-\vert \nu\vert}\,,& \mbox{if }\nu\neq 0 \end{cases}
\,,\qquad \mbox{as $t \to 0^+$}\,,
 \end{equation}
 where $\Gamma(\nu)$ is the Gamma function and $\gamma$ is the Euler-Mascheroni constant, see \cite[Eqs. 10.27.1-3 and \S 10.30]{NIST}. As a consequence, for any $k \in \mathbb{Z}$, $K_{k},K_{k+1}$ are never in $L^2((0,R),rdr)$ \emph{at the same time}, while this is always the case for $I_k, I_{k+1}$. In light of this, we must fix
 \[
 	b^{+}_k = 0\,.
 \]

On the other side, we recall the large argument asymptotic expansions \cite[Eqs. \S 10.30]{NIST}
\begin{equation}
I_\nu(t)\sim \frac{1}{\sqrt{2\pi t}}\,e^t\,, \qquad 
K_\nu(t)\sim \sqrt{\frac{\pi}{2t}}\, e^{-t}\,, \qquad 
\quad \mbox{as\; $t \to +\infty$}\,. \label{eq:sasyminf}
\end{equation}
Considering this, a non-trivial solution which is square-integrable at infinity can be obtained only if $z \in (-m,m)$. Incidentally, this accounts for the fact that the discrete spectrum is contained in the open gap interval. For $z \in (-m,m)$, the functions $K_k,K_{k+1}$ are simultaneously in $L^2((R,+\infty), r dr)$, while $I_k,I_{k+1}$ are not. Accordingly, we have to fix
\[
	a^{-}_k = 0\,.
\]

Summing up, square-integrable solutions of \eqref{eq:eveqk} are of the form $\psi_k = (u_k,v_k)$ with
\begin{equation}\label{eq:uvIKproof}
\begin{cases}
u_k(r) = \one_{(0,R)}(r)\, a^{+}_k\, I_k(\sqrt{m^2-z^2}\,r) + \one_{(R,+\infty)}(r)\, b^{-}_k\, K_k(\sqrt{m^2-z^2}\,r) \,, \vspace{0.1cm}\\
v_k(r) = -i\,\sqrt{\frac{m-z}{m+z}} \left[\one_{(0,R)}(r)\,a^{+}_k\, I_{k+1}(\sqrt{m^2-z^2}\,r) - \one_{(R,+\infty)}(r)\,b^{-}_k\, K_{k+1}(\sqrt{m^2-z^2}\,r)\right] ,
\end{cases}
\end{equation}
for suitable $a^{+}_k,b^{-}_k \in \CC$.

{\sl 3) The boundary conditions at $r = R$.}
For any solution of the form \eqref{eq:uvIKproof}, the boundary condition \eqref{eq:bcradial} at $r = R$ can be rephrased as:
\begin{equation}\label{eq:bcsystem}
\begin{pmatrix}
\frac{\eta+\tau}{2}I_k(t) + \sqrt{\frac{m-z}{m+z}} I_{k+1}(t)  & \frac{\eta+\tau}{2}K_k(t)+ \sqrt{\frac{m-z}{m+z}} K_{k+1}(t)  \\
\frac{\eta-\tau}{2}\sqrt{\frac{m-z}{m+z}}I_{k+1}(t)-I_k(t) & -\frac{\eta-\tau}{2}\sqrt{\frac{m-z}{m+z}}K_{k+1}(t)+K_k(t)
\end{pmatrix}_{t \,=\, \sqrt{m^2-z^2}\,R} 
\begin{pmatrix} a^{+}_k \\ b^{-}_k \end{pmatrix} = 0\,.
\end{equation}
Also in this case, a non-trivial solution exists if and only if the determinant of the square matrix on the left-hand side vanishes. This condition ultimately boils down to
\begin{equation}\label{eq:evcond}
\left[(\eta+\tau)\,I_k(t)\, K_k(t) - (\eta-\tau)\frac{m-z}{m+z}\, I_{k+1}(t)\,K_{k+1}(t)\right]_{t \,=\, \sqrt{m^2-z^2}\,R} = 0\,,
\end{equation}
which appears to be equivalent to \eqref{eq:evid}. The corresponding solution matches the choice
\begin{align*}
 b^{-}_k 
= a^{+}_k \left[\frac{\frac{\eta-\tau}{2}\sqrt{\frac{m-z}{m+z}}\,I_{k+1}(t)-I_k(t)}{\frac{\eta-\tau}{2}\sqrt{\frac{m-z}{m+z}}\,K_{k+1}(t)-K_k(t)}\right]_{t \,=\, \sqrt{m^2-z^2}\,R} .
\end{align*}
which, together with \eqref{eq:uvIKproof}, accounts for \eqref{eq:formk}, \eqref{eq:uvIK} and \eqref{eq:const}.

{\sl 4) The normalization factor.} Assume $z_k$ is a solution of \eqref{eq:evid} for a given $k \in \Z$, and consider the associated eigenfunction defined by \eqref{eq:formk}, \eqref{eq:uvIK} and \eqref{eq:const}. Requiring that $\psi_k$ has unit norm in $L^2(\R^2;\CC^2)$ amounts to
\begin{align*}
1 
& = \int_{\R_+ \times [0,2\pi)}\vert \psi_k(r,\theta)\vert^2\,r dr d\theta 
= \int_{0}^{\infty} \left(\vert u_k(r)\vert^2+\vert v_k(r)\vert^2\right)\,r dr \\
& = |a_k|^2 \bigg[ \int_{0}^{R} \left|I_k\left(\sqrt{m^2-z_k^2}\,r\right)\right|^2\! r dr + |c_{k}|^2 \int_{R}^{\infty} \left|K_k\left(\sqrt{m^2-z_k^2}\,r\right)\right|^2\! r dr \\
& \hspace{2.2cm} + \frac{m-z_k}{m+z_k} \int_{0}^{R} \left|I_{k+1}\left(\sqrt{m^2-z_k^2}\,r\right)\right|^2\!r dr + \frac{m-z_k}{m+z_k}\,|c_{k}|^2 \int_{R}^{\infty} \left|K_{k+1}\left(\sqrt{m^2-z_k^2}\,r\right)\right|^2\!r dr \bigg] \\
& = \frac{|a_k|^2}{m^2-z_k^2} \bigg[ \int_{0}^{t_k} |I_k(t)|^2\, t\, dt + |c_{k}|^2 \int_{t_k}^{\infty} |K_k(t)|^2\, t\, dt \\
& \hspace{2.2cm} + \frac{m-z_k}{m+z_k} \int_{0}^{t_k} |I_{k+1}(t)|^2\, t\, dt + \frac{m-z_k}{m+z_k}\,|c_{k}|^2 \int_{t_k}^{\infty} |K_{k+1}(t)|^2\, t\, dt \bigg]_{t_k \,=\,\sqrt{m^2-z_k^2}\,R}.
\end{align*}
Now, let us notice that the Bessel functions $I_\nu,K_\nu$ are positive definite \cite[\S 10.37]{NIST} and recall the following relation, where $Z_\nu$ denotes an arbitrary modified Bessel function of order $\nu$ \cite[5.54.2]{GR07}:
\[
\int \,[Z_\nu(t)]^2\, t\, dt =\frac{t^2}{2}\, \Big[ Z_\nu^2(t) - Z_{\nu-1}(t)Z_{\nu+1}(t) \Big]\,.
\]
From here and from the previously mentioned asymptotic expansions \eqref{eq:sasym}, \eqref{eq:sasyminf}, we infer
\begin{multline*}
1 
= \frac{R^2}{2}\, |a_k|^2\, \bigg[
I_k^2(t_k) - I_{k-1}(t_k)I_{k+1}(t_k)
- |c_{k}|^2 \Big(K_k^2(t_k) - K_{k-1}(t_k) K_{k+1}(t_k) \Big)
 \\
+ \frac{m-z_k}{m+z_k} \left[ I_{k+1}^2(t_k) - I_{k}(t_k R)I_{k+2}(t_k)
- |c_{k}|^2 \Big( K_{k+1}^2(t_k) - K_{k}(t_k)K_{k+2}(t_k) \Big)\right]
 \bigg]_{t_k \,=\,\sqrt{m^2-z_k^2}\,R}\,,
\end{multline*}
which motivates the choice of $a_k$ reported in \eqref{eq:akunit}.
\end{proof}

\begin{lemma}\label{lem:noev}
Under the assumptions of \cref{thm:maincirc}, the spectral point
\[
\zet=-\frac{\tau}{\eta}\, m\in\sigma_{\mathrm{ess}}(\DDet)\,,
\]
is \emph{not} an eigenvalue.
\end{lemma}

\begin{proof}
We argue by contradiction, assuming that $\zet=-\frac{\tau}{\eta}\, m$ is an eigenvalue of $\DDet$. As a consequence of the preceding \cref{lem:implicit}, the identity \eqref{eq:evid} must then hold with $z = \zet$. Since
\[
\frac{(\eta - \tau)(m-\zet)}{(\eta + \tau)(m+\zet)} = 1\,,
\]
in the case under analysis the cited identity reduces to 
\begin{equation}\label{eq:idbessel}
\big[I_{k+1}(t)K_{k+1}(t)-I_{k}(t)K_{k}(t)\big]_{t\,=\,\sqrt{m^2-\zet^2}\, R} = 0\,.
\end{equation}
On the other hand, the {\sl strict} inequality $I_{k+1}(t)/I_k(t) < K_{k}(t)/K_{k+1}(t)$ for all $k \geqslant 0$ and $t > 0$ \cite[Thm. 1]{Se11}, together with the connection formulas \cite[Eq. 10.27.1 and 10.27.3]{NIST}
\begin{equation}\label{eq:IKconn}
	I_{k}(t) = I_{-k}(t)\,, \qquad K_{k}(t) = K_{-k}(t)\,, \qquad \mbox{for $k \in \Z$}\,,
\end{equation}
imply
\begin{align*}
I_{k+1}(t)K_{k+1}(t) - I_{k}(t)K_{k}(t) < 0\,, \qquad & \mbox{for all $k \in \N$ and $t > 0$}\,, \\
I_{k+1}(t)K_{k+1}(t) - I_{k}(t)K_{k}(t) > 0\,, \qquad & \mbox{for all $k \in \Z \setminus \N$ and $t > 0$}\,.
\end{align*}
In both cases, the equality \eqref{eq:idbessel} fails to hold. This contradiction concludes the proof.
\end{proof}

Let us now report a technical result, which will be later employed to characterize the solutions $z_k \in (-m,m)$ of the implicit relation \eqref{eq:evid}.

	\begin{proposition}\label{prop:fg}
		Let $a > 0$ and $p \in (-1,1)$. For $t \in (0,a)$ and $k \in \Z$, consider the functions defined as
		\begin{equation}\label{eq:fg}
			f_k(t):=\frac{I_{k+1}(t)K_{k+1}(t)}{I_{k}(t)K_{k}(t)}\,,\qquad
			g_{p,\pm}(t):= \frac{1-p}{1+p}\;\frac{1\pm\sqrt{1-(t/a)^2}}{1\mp\sqrt{1-(t/a)^2}} \,.
		\end{equation}
		Then, there exists $n_0 > 0$ such that, for every $k\in\Z$ with $|k| \geqslant n_0$, 
		
			there exists $t_k \in[0,a)$
		satisfying
		\begin{equation}
			f_k(t_k)= \begin{cases}
				g_{p,-}(t_k) & \mbox{if\; $p \in (-1,0]$}\,,\\
				g_{p,+}(t_k) & \mbox{if\; $p \in (0,1)$}\,, 
			\end{cases}
		\end{equation}
		Moreover,
		\begin{equation}\label{eq:conv}
			\lim_{k\to \pm \infty} t_k = a\,\sqrt{1-p^2}\,.
		\end{equation}
	\end{proposition}

	\begin{figure}[t!]
    \begin{minipage}[b]{0.45\textwidth}
        \centering
        \includegraphics[width=\textwidth]{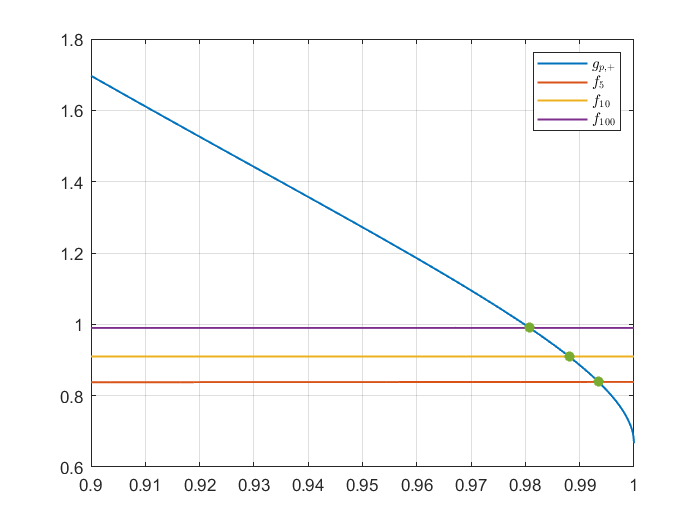}
    \end{minipage}
    \begin{minipage}[b]{0.45\textwidth}
        \centering
        \includegraphics[width=\textwidth]{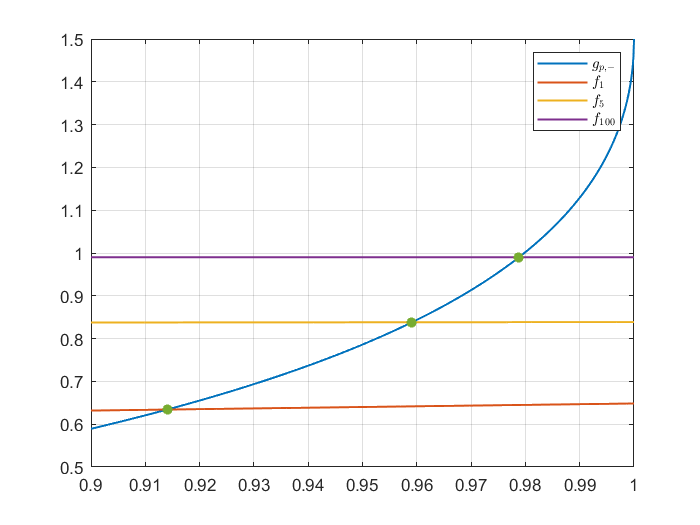}
    \end{minipage}
    \caption{
    			Plots of the functions $f_k,g_{p,\pm}$ introduced in \cref{prop:fg}.
			The left panel displays $f_k$, for $k = 5,10,100$, together with $g_{p,+}$, for $a=1$, $p = 0.2$.
			The right panel shows $f_k$, for $k = 1,5,100$, together with $g_{p,-}$, for $a=1$, $p = -0.2$.
			Green dots mark the intersection points, where $f_k = g_{p,\pm}$.
		}
    \label{fig:fkgp}
	\end{figure}

\begin{proof}
As a preliminary step, we notice that the connection formulas \eqref{eq:IKconn} imply
\[
f_{-k-1}(t) = \frac{1}{f_{k}(t)} \qquad \mbox{for all $k \in \Z$}\,.
\]
Furthermore, it is straightforward to verify that
\[
g_{-p,\mp}(t) = \frac{1}{g_{p,\pm}(t)} \qquad \mbox{for all $p \in (-1,1)$}\,.
\]
On account of these identities, it suffices to prove the thesis for $k \geqslant 0$, since the case $k \leqslant -1$ then follows straightforwardly. So, we henceforth fix
\[
k = n \in \N\,.
\]

Let us establish some basic properties of the function $f_n$ defined in \eqref{eq:fg}.
\begin{enumerate}[1)]

	\item Considering that the Bessel functions $I_n(t),K_n(t)$ are positive definite for $n \in \N$ and $t > 0$, and recalling the previously mentioned inequality $I_{n+1}(t)/I_n(t) < K_{n}(t)/K_{n+1}(t)$, we readily obtain
		\[
			0<f_n(t)<1\,,\qquad \mbox{for all $t > 0$}\,.
		\]

	\item Using the asymptotics \eqref{eq:sasym} \eqref{eq:sasyminf}, together with the basic identity $\Gamma(n+1)=n!$ \cite[Eq. 5.4.1]{NIST}, we infer
		\[
			\lim_{t \to 0^+} f_n(t)=\frac{n}{n+1} \,, \qquad 
			\lim_{t \to +\infty} f_n(t) = 1\,.
		\] 

	\item Using the identities \cite[Eq. 10.29.4]{NIST}
		\[
			\frac{1}{t}\frac{d}{dt}\big[t^{-n} I_n(t) \big] = t^{-n-1} I_{n+1}(t)\,,\quad 
			\frac{1}{t}\frac{d}{dt}\big[t^{n} K_n(t) \big] = - t^{n-1} K_{n-1}(t)\,,
		\]
		by elementary calculations we obtain
		\begin{align*}
			\frac{d}{dt}f_n&=\frac{d}{dt}\left[\frac{(t^{-(n+1)}I_{n+1})\,(t^{n+1}K_{n+1})}{(t^{-n}I_n)\,(t^n K_n)} \right] \\
& =\frac{1}{I^2_n\,K^2_n}\,\Big[\Big(I_n\, I_{n+2} - I^2_{n+1}\Big)\, K_n\, K_{n+1} + \Big(K_{n-1}\,K_{n+1} - K^2_n\Big)\,I_n\, I_{n+1} \Big]\,,
		\end{align*}
		where we omitted the arguments of the Bessel functions for notation convenience.
		Then, exploiting the Tur\'an type inequalities \cite[Thms. 1,2 and 4]{Ba15}
		\begin{align*}
			I_n(t) I_{n+2}(t)-I^2_{n+1}(t) & \geq - \frac{I^2_{n+1}(t)}{\sqrt{t^2+(n+1)^2-1/4}}\,, \\
			K_{n-1}(t)K_{n+1}(t)-K^2_{n}(t)&\geq \frac{K^2_{n}(t)}{\sqrt{t^2+\max(0,n^2-1/4)}}\,,
		\end{align*}
		and using once more that $I_{n+1}(t)/I_n(t) < K_n(t)/K_{n+1}(t)$, we obtain
		\begin{align*}
			\frac{d}{dt}f_n(t)
			& \geqslant \frac{I_{n+1}K_n}{I^2_n K^2_n} \left[-\frac{I_{n+1} K_{n+1}}{\sqrt{t^2+(n+1)^2-1/4}} + \frac{I_n K_n}{\sqrt{t^2+\max(0,n^2-1/4)}} \right] \\
			& \geqslant \frac{I_{n+1}K_n}{I^2_n K^2_n}\, I_{n+1}K_{n+1} \left[\frac{1}{\sqrt{t^2+\max(0,n^2-1/4)}}-\frac{1}{\sqrt{t^2+(n+1)^2-1/4}} \right] > 0 \,.
		\end{align*}
		This shows that $f_n$ is strictly monotonically increasing.
\end{enumerate} 
In particular, combining items $ii)$ and $iii)$ above, we obtain
\[
\lim_{n\to+\infty}\sup_{t\in[0,+\infty)}\vert 1-f_n(t)\vert =\lim_{n\to+\infty} \vert 1- f_n(0^+)\vert 
= \lim_{n\to+\infty}\left\vert 1- \frac{n}{n+1}\right\vert=0\,,
\]
which proves that
\begin{equation}\label{eq:fkconv}
f_n \to 1\,,\qquad \mbox{as $n\to+\infty$}\,,
\end{equation}
and the convergence is uniform on any compact subset of $[0,+\infty)$.
\smallskip

We now proceed to examine the functions $g_{p,\pm}$, defined in \eqref{eq:fg}.
\\
Let us first assume $p \leqslant 0$. Then, we have
\[
\lim_{t \to 0^+} g_{p,-}(t) = 0 \leqslant \frac{n}{n+1}=\inf_{t \in[0,a)}f_n(t)\,, \qquad \mbox{for all $n \in \N$}\,,
\]
and 
\[
	\lim_{t \to a^-} g_{p,-}(t) = \frac{1+|p|}{1-|p|} \geqslant 1 > \sup_{t\in[0,a)}f_n(t)\,.
\]
Furthermore, it can be easily checked that $g'_{p,-} > 0$, which ensures that $g_{p,-}$ is strictly monotonically increasing.
Together with the continuity of $g_{p,-}$ and the previously established features of $f_n$, this implies that for each $n \in \N$ there exists $t_n \in [0,a)$
such that
\[
f_n(t_n) = g_{p,-}(t_n)\,.
\]
	Since both $f_n$ and $g_{p,-}$ are increasing functions, multiple intersections may occur depending on their concavity properties; we do not investigate this issue further. On the other hand,
building on \eqref{eq:fkconv}, we infer that
\[
1 = \lim_{n\to+\infty}f_n(t_n) = \lim_{n \to +\infty}g_{p,-}(t_n) = g_{p,-}\!\left(\lim_{n\to+\infty} t_n\right) .
\]
Observe that the sequence $\{t_n\}_{n\in\N}$ admits a limit $t_\infty$ because this is the case for $\{g_{p,-}(t_n)\}_{n \in \N}$ and $g_{p,-}$ is a monotone function. In particular,
\[
g_{p,-}(t_\infty) = 1 \qquad \iff \qquad t_\infty = a\,\sqrt{1-p^2} \,.
\]

The case $p > 0$ follows by similar arguments, considering the monotone decreasing function $g_{p,+}$ and observing that
\[
\lim_{t \to 0^+} g_{p,+}(t)=+\infty \,,\qquad 
\lim_{t \to a^-} g_{p,+}(t) = \frac{1-p}{1+p} < 1\,.
\]
In particular, we have
\[
\lim_{t \to a^-} g_{p,+}(t) < \frac{n}{n+1} = \inf_{t \in[0,a)}f_n(t)\,, \qquad
\mbox{for any $n \in \N$ with $n > n_0 := \left\lceil \frac{1}{2}\!\left(\frac{1}{p}- 1\right) \right\rceil$}\,,
\]
which yields the thesis. In this case, the intersection point $t_n$ is unique (whenever it exists), since $g_{p,+}$ is decreasing whereas the function $f_n$ is increasing.
\end{proof}

We can now prove \cref{thm:maincirc}.

\begin{proof}[Proof of \cref{thm:maincirc}]
\cref{lem:implicit} accounts for all the stated claims except for the existence of solutions to \eqref{eq:evid} and the accumulation property \eqref{eq:zacc}. Both of these remaining points follow from \cref{prop:fg}, as explained below.

For $z \in (-m,m)$, we set 
\[
t := \sqrt{m^2-z^2}\,R \in [0,m R)\,,
\]
so that
\[
z = \pm m\, \sqrt{1-\left(\frac{t}{mR} \right)^2}\,.
\]
With this change of variable, \eqref{eq:evid} can be recast in the form
\begin{equation*}
\frac{I_{k+1}(t)K_{k+1}(t)}{I_{k}(t)K_{k}(t)} 
= \frac{1+\frac{\tau}{\eta}}{1-\frac{\tau}{\eta}}\,\frac{1 \pm \sqrt{1-\left(\frac{t}{mR} \right)^2}}{1 \mp \sqrt{1-\left(\frac{t}{mR} \right)^2}}\,.
\end{equation*}
Then, the thesis follows by a straightforward application of \cref{prop:fg}, here employed with $p = \zet/m = - \tau/\eta \in (-1,1)$ and $a = m R > 0$.
\end{proof}

Let us now derive the asymptotic properties outlined in \cref{cor:asef}.

\begin{proof}[Proof of \cref{cor:asef}]
We first examine the behavior of the eigenvalues $\{z_k\}_{k \in \Z}$ as $k \to \pm \infty$. To this purpose, it is convenient to rephrase the defining equation \eqref{eq:evid} as
\begin{equation}\label{eq:W0}
\left[\frac{I_{k+1}(t)\, K_{k+1}(t)}{I_{k}(t)\, K_{k}(t)}\right]_{t \,=\, \sqrt{1-(\frac{z_k}{m})^2}\,m R} = \frac{(1+\frac{\tau}{\eta})(1+ \frac{z_k}{m})}{(1-\frac{\tau}{\eta})(1 - \frac{z_k}{m})} \,.
\end{equation}
From \eqref{eq:zacc}, we immediately deduce that
\begin{equation*}
	\left(1-\frac{\tau}{\eta}\right) \lim_{k \to \pm \infty}\! \left(1 - \frac{z_k}{m}\right)
	= \left(1+\frac{\tau}{\eta}\right) \lim_{k \to \pm \infty}\! \left(1 + \frac{z_k}{m}\right)
	= \lim_{k \to \pm \infty}\! \left(1 - \left(\frac{z_k}{m}\right)^2\right) 
	= \frac{4}{\eta^2}\,.
\end{equation*}
On the other side, let us recall the uniform asymptotic expansions of the Bessel functions $I_k(t),K_k(t)$ for $t > 0$ ranging in a compact set and $k \to +\infty$, see \cite{SH11} or \cite[\S 10.41(ii)]{NIST}:
\begin{align}
	I_k(t) &= \frac{1}{\sqrt{2\pi k}}\!\left(\frac{e t}{2k}\right)^{\!k} \left[ 1 - \frac{1\!-\!3t^2}{12\, k} + \frac{1 \!-\! 78 t^2 \!+\! 9 t^4}{288\, k^2} + \frac{139 \!+\! 14085 t^2 \!-\! 4995 t^4 \!+\! 135 t^6}{51840\, k^3} + O\!\left(\frac{1}{k^4}\right) \right] ; \label{eq:Ikasy}\\
	K_k(t) &= \sqrt{\frac{\pi}{2 k}}\! \left(\frac{e t}{2k}\right)^{\!-k} \left[  1 + \frac{1\!-\!3t^2}{12\, k} + \frac{1 \!-\! 78 t^2 \!+\! 9 t^4}{288\, k^2} - \frac{139 \!+\!14085 t^2 \!-\! 4995 t^4 \!+\! 135 t^6}{51840\, k^3} + O\!\left(\frac{1}{k^4}\right) \right] . \label{eq:Kkasy}
\end{align}
Using these relations, together with the connection formulas \eqref{eq:IKconn} when $k \leqslant -1$, a direct computation yields
\begin{equation}
	\frac{I_{k+1}(t)\, K_{k+1}(t)}{I_{k}(t)\, K_{k}(t)} = 1 - \frac{1}{k} + \frac{1}{k^2} - \frac{1-4t^2}{k^3} + O\!\left(\frac{1}{k^4}\right)\,, \qquad \mbox{as $k \to \pm \infty$}\,. \label{eq:asyfrac}
\end{equation}
We now assume that $z_k$ admits an asymptotic expansion in integer powers of $1/k$, namely,
\begin{equation}
		z_k = \sum_{\ell = 0}^{3} \zeta^{\pm}_\ell\,\frac{1}{k^\ell} + O\!\left(\frac{1}{k^4}\right), \qquad \mbox{as $k \to \pm \infty$}\,,\label{eq:zkser}
\end{equation}
for suitable coefficients $\zeta_\ell \in \R$. Substituting the above expansions \eqref{eq:asyfrac} \eqref{eq:zkser} in \eqref{eq:W0} and matching powers of $1/k$, we can solve iteratively for the coefficients $\zeta_\ell$. This procedure yields \eqref{eq:evas}.
\smallskip 

We now turn to the normalized eigenfunctions $\psi_k$ defined by \eqref{eq:formk}, \eqref{eq:uvIK}, \eqref{eq:const} and \eqref{eq:akunit}.
Let us observe that the expansions \eqref{eq:Ikasy}\eqref{eq:Kkasy} imply, in particular,
\begin{eqnarray*}
& I_{k + 1}(t) \sim \left(\frac{t}{2|k|}\right)^{\pm 1}\! I_k(t)\,,\qquad K_{k + 1}(t)\sim \left(\frac{2|k|}{t}\right)^{\pm 1}\! K_k(t)\,, \\
& I_k^2(t_k) - I_{k-1}(t_k)I_{k+1}(t_k) \sim \frac{1}{2 \pi |k|^2}\,e^{\mp\tau/\eta} \left(\frac{|\eta|\,|k|}{e m R}\right)^{\!-2|k|} , \\
& K_{k-1}(t_k) K_{k+1}(t_k) - K^2_k(t_k) \sim \frac{\pi}{2 |k|^2}\,e^{\pm \tau/\eta} \left(\frac{|\eta|\, |k|}{e m R}\right)^{\!2|k|} ,
\end{eqnarray*}
as\, $k \to \pm\infty$. Using these relations, together with $\eta^2-\tau^2 = 4$ and $z_k/m \sim -\tau/\eta$ for $k \to \pm \infty$, we infer that the coefficients $c_{k}$ and $a_k$, defined in \eqref{eq:const} and \eqref{eq:akunit}, respectively, satisfy the following:
\begin{gather*}
c_{k} 
\sim \left[\frac{\sgn(\eta) \,I_{k+1}(t) - I_k(t)}{\sgn(\eta)\,K_{k+1}(t)-K_k(t)}\right]_{t\,=\,2mR/|\eta|}
\sim \mp\,\sgn(\eta) \,\frac{e^{\mp 1}}{\pi} \left(\frac{e m R}{|\eta|\, |k|}\right)^{\!2|k|\pm 1};
\\
a_k \sim \sqrt{2\pi}\,L_{\eta,\tau,k}\,|k| \left(\frac{|\eta|\,|k|}{e m R}\right)^{\!|k|}\,,
\end{gather*}
where we put, for brevity,
\[
L_{\eta,\tau,k} := \frac{\sqrt{2}}{R}\left[ e^{-\tau/\eta} + \frac{\eta+\tau}{\eta-\tau}\, e^{\tau/\eta} \right]^{-1/2}\times \begin{cases}
1 \,, & \mbox{for $k \to + \infty$}\,, \\
\frac{m R}{|\eta|\,|k|} \,, & \mbox{for $k \to - \infty$}\,.
\end{cases}
\]
Plugging the above asymptotics into the expression \eqref{eq:uvIK}, for $k \to \pm \infty$ we get
\begin{align*}
\begin{cases}
u_k(r) \sim L_{\eta,\tau,k} \,\sqrt{|k|}
\left[\one_{(0,R)}(r)\, \!\left(\frac{r}{R}\right)^{\!|k|}
\mp \sgn(\eta) \,\one_{(R,+\infty)}(r) \left(\frac{|\eta|\, |k|}{m R}\right)^{\!\mp 1}\! \left(\frac{R}{r}\right)^{\!|k|}
\right] ,\vspace{0.1cm} \\
v_k(r) \sim -i\, L_{\eta,\tau,k}\,\sqrt{|k|}\,\sqrt{\frac{\eta+\tau}{\eta-\tau}} \left[
\one_{(0,R)}(r)  \left(\frac{m R}{|\eta| |k|}\right)^{\!\pm 1}\! \left(\frac{r}{R}\right)^{\!|k|\pm 1} 
 \pm\sgn(\eta)\, \one_{(R,+\infty)}(r) \left(\frac{R}{r}\right)^{\!|k| \pm 1 } \right] .
 \end{cases}
\end{align*}
Considering the exponential decay of the ratios $(r/R)^{|k|}$ and $(R/r)^{|k|}$ for $r<R$ and $r>R$, respectively, the above expansions ultimately prove that $\psi_k(r,\theta)$ vanishes pointwisely for any fixed $r \in (0,+\infty) \setminus \{R\}$ as $k \to \pm \infty$.
\end{proof}

We now conclude with the proof of \cref{cor:symev}.

\begin{proof}[Proof of \cref{cor:symev}]
Consider the bounded operators $\Con$ and $\UU$ defined in \eqref{eq:cc} and \eqref{eq:U}, respectively. We establish hereafter the identity \eqref{eq:CUH}, from which the symmetry statement in \eqref{eq:symz} follows immediately.

To this end, we need to examine the action of $\Con$ and $\UU$ on the boundary conditions characterizing the domain of $\DDet$, see \eqref{eq:Ddom}, and on the underlying differential expression, see \eqref{eq:Daction}. In what follows, we denote by $\psi = \psi_+ \oplus \psi_-$ a generic spinor in $\dom(H_{\eta,0})$.

On one hand, noting that $\sigma_1$ has real entries, while $\sigma_2$ is purely imaginary, and that $\sigma_2 \sigma_1 = - \sigma_1 \sigma_2$, the boundary conditions in \eqref{eq:Ddom} yield
\begin{multline}
-i(\nv \cdot \siv)\big[\TD_+\mathcal (\Con \psi)_+-\TD_-(\Con \psi)_-\big]
=-i(n_1 \sigma_1 + n_2 \sigma_2)(\TD_+ (\sigma_1\overline\psi)_+-\TD_-(\sigma_1\overline\psi)_-) 
\\
= \sigma_1(-in_1\sigma_1+i n_2\sigma_2)(\TD_+ \overline\psi_+-\TD_-\overline\psi_-)
= - \sigma_1\overline{(- i \nv \cdot \siv)(\TD_+ \psi_+-\TD_-\psi_-) }
\\
= -\sigma_1\,\tfrac{\eta}{2}\,\overline{(\TD_+ \psi_+ +\TD_-\psi_-)}
= -\tfrac{\eta}{2}\,\big[\TD_+(\Con\psi)_+ +\TD_-(\Con\psi)_-\big]\,.
\end{multline}
Similar arguments, can be used to prove that $\Con$ formally anti-commutes with the differential expression \eqref{eq:Daction}, that is
\begin{multline}
(-i(\sigma_1\partial_x+\sigma_2\partial_y)+m\sigma_3)(\Con\psi)_+ \oplus(-i(\sigma_1\partial_x+\sigma_2\partial_y)+m\sigma_3)(\Con\psi)_-
\\ 
= -\,\Con\left[ (-i(\sigma_1\partial_x+\sigma_2\partial_y)+m\sigma_3)\psi_+\oplus(-i(\sigma_1\partial_x+\sigma_2\partial_y)+m\sigma_3)\psi_-\right] .
\end{multline}
The above results yield
\begin{equation}\label{eq:CD}
	\Con \,\DDet = - \DD_{-\eta,0}\, \Con\,.
\end{equation}

Let us now deal with the operator $\UU$. Using again basic properties of the Pauli matrices $\sigma_1,\sigma_2$ (notice, in particular, that $(-i\, \nv \cdot \siv)^2 = -\one$, together with the boundary condition in \eqref{eq:Ddom}, by direct calculations we obtain
\begin{multline}
-i(\nv \cdot \siv) \big[\TD_+ (\UU\psi)_+ - \TD_-(\UU\psi)_-\big]
= -i(\nv \cdot \siv)(\TD_+\psi_+ + \TD_-\psi_-) 
\\
= \frac{2}{\eta}\,(-i\nv \cdot \siv)^2 (\TD_+\psi_+ - \TD_-\psi_-) 
= -\frac{2}{\eta}\,\big[(\TD_+(\UU\psi)_+ + \TD_-(\UU\psi)_-)\big] 
\\
= -\frac{\eta}{2}\,\big[(\TD_+(\UU\psi)_+ + \TD_-(\UU\psi)_-)\big]\,,
\end{multline}
where the last identity follows noting that, for critical interaction strengths $\eta = \pm 2$ and $\tau = 0$, there holds
\[
2/\eta = \eta/2\,.
\]
In addition, it is evident that $\UU$ formally commutes with the differential expression \eqref{eq:Daction}. Therefore, we have
\begin{equation}\label{eq:UD}
	\UU \,\DD_{\eta,0} = \DD_{-\eta,0} \,\UU \,.
\end{equation}

Combining the identities \eqref{eq:CD} and \eqref{eq:UD}, we finally obtain \eqref{eq:CUH}, whence the thesis.
\end{proof}


\section*{Acknowledgements}
We thank Jussi Behrndt for valuable insights and discussions related to this work.
\smallskip

\noindent
Numerical simulations and plots were generated using \texttt{MATLAB} (The MathWorks Inc., Natick, MA).

\section*{Funding}
The authors have been supported by MUR grant \emph{Dipartimento di Eccellenza 2023-2027} of Dipartimento di Matematica at Politecnico di Milano. W.B. is member of Gruppo Nazionale per l'Analisi Matematica, la Probabilità e le loro Applicazioni GNAMPA – INdAM in Italy. W.B. acknowledges that this study was carried out within the project E53D23005450006 ``Nonlinear dispersive equations in presence of singularities'' -- funded by European Union -- Next Generation EU within the PRIN 2022 program (D.D. 104 - 02/02/2022 Ministero dell'Universit\`a e della Ricerca).

\section*{Conflict of interest}
The author declares no conflict of interest.

\section*{Data availability statement}
Data sharing is not applicable to this article as it has no associated data.


\end{document}